\UseRawInputEncoding
\documentclass[11pt, reqno, a4paper]{article}
\usepackage[dvipsnames]{xcolor}
\usepackage{amsmath,amssymb,amsthm,url,mathtools}

\usepackage[lmargin=35mm,rmargin=35mm,bmargin=30mm,tmargin=30mm]{geometry}
\usepackage{bbm}
\usepackage[inline]{enumitem}
\usepackage{subcaption}

\usepackage{esvect} 
\usepackage{hyperref}
\usepackage{parskip}

\usepackage[longnamesfirst,numbers,sort]{natbib}

\hypersetup{
  colorlinks   = true,
  urlcolor     = blue!50!black,
  linkcolor    = blue!50!black,
  citecolor   = blue!50!black
}

\usepackage{tikz}
\usetikzlibrary{snakes}
\usetikzlibrary{decorations.pathmorphing}

\title{Balancing connected colourings of graphs}
\author{
Freddie Illingworth\footnote{Mathematical Institute, University of Oxford,
Oxford OX2\thinspace6GG, United Kingdom}~\footnote{Research supported by EPSRC grant EP/V007327/1} \;
Emil Powierski\protect\footnotemark[1] \;
\\Alex Scott\protect\footnotemark[1]~\footnotemark[2] \;
Youri Tamitegama\protect\footnotemark[1]}
\date{}

\newtheorem{theorem}{Theorem}
\newtheorem{conjecture}[theorem]{Conjecture}

\newtheorem{claim}{Claim}
\newtheorem*{claim*}{Claim}
\newtheorem{lemma}{Lemma}[section]

\newtheorem{question}{Question}
\theoremstyle{definition}
\newtheorem{definition}{Definition}
\newtheorem*{definition*}{Definition}

\numberwithin{equation}{section}

\newcommand{\N}{\mathbb{N}}
\newcommand{\Q}{\mathbb{Q}}
\renewcommand{\leq}{\leqslant} 
\renewcommand{\geq}{\geqslant}

\newcommand\blfootnote[1]{
  \begingroup
  \renewcommand\thefootnote{}\footnote{#1}
  \addtocounter{footnote}{-1}
  \endgroup
}
\advance\textwidth2\evensidemargin
\begin{document}
\maketitle
\begin{abstract}
We show that the edges of any graph $G$ containing two edge-disjoint spanning trees can be blue/red coloured so that the blue and red graphs are connected and the blue and red degrees at each vertex differ by at most four.
This improves a result of H\"orsch.
We discuss variations of the question for digraphs, infinite graphs and a computational question, and resolve two further questions of H\"{o}rsch in the negative.
\blfootnote{Email: \texttt{\{\href{mailto:illingworth@maths.ox.ac.uk}{illingworth},\href{mailto:powierski@maths.ox.ac.uk}{powierski},\href{mailto:scott@maths.ox.ac.uk}{scott},\href{mailto:tamitegama@maths.ox.ac.uk}{tamitegama}\}@maths.ox.ac.uk}}
\end{abstract}

\section{Introduction}

Finding edge-disjoint spanning trees in a graph has a rich history. 
The seminal result is the independent characterisation by Tutte~\cite{tutte1961problem} and Nash-Williams~\cite{nash1961edge} of the presence of $k$ edge-disjoint spanning trees in a finite graph. Much research has focussed on whether the packed spanning trees can be chosen to satisfy extra properties (for example, see \cite{bang2016complexity,bessy2021fpt,chuzhoy2020packing}). 
It is folklore that the edges of any graph $G$ can be coloured blue and red such that the blue degree and red degree of each vertex differ by at most two. 
The intersection of these two problems asks how well the colour-degrees can be balanced in a blue/red-edge colouring of a graph that contains a \emph{double tree} -- the union of two edge-disjoint spanning trees. Kriesell~\cite{kriesell2011balancing} was the first to consider balancing colour-degrees in a blue/red-edge colouring of a double tree.
Building on his work, H\"{o}rsch~\cite{horsch2021globally} gave the first constant bound when $G$ is a double tree.

\begin{theorem}[H\"{o}rsch~\cite{horsch2021globally}]\label{thm:horsch}
	Let $G$ be a finite double tree. 
	The edges of $G$ may be coloured blue and red such that the blue and red graphs are both spanning trees and the blue and red degrees of each vertex differ by at most five.
\end{theorem}

Our main result is two-fold. Firstly we reduce the above bound to four.

\begin{theorem}\label{thm:2trees}
Let $G$ be a finite double tree. 
The edges of $G$ may be coloured blue and red such that the blue and red graphs are both spanning trees and the blue and red degrees of each vertex differ by at most four.
\end{theorem}
Moreover, we obtain the same bound in the more general case where $G$ is any graph containing a spanning double tree.

\begin{theorem}\label{thm:2treesgeneral}
Let $G$ be a finite graph containing a spanning double tree. 
The edges of $G$ may be coloured blue and red such that the blue and red graphs both contain spanning trees and the blue and red degrees of each vertex differ by at most four.
\end{theorem}

H\"{o}rsch asked whether Theorem~\ref{thm:horsch} can be extended to infinite graphs. 
The Tutte-Nash-Williams characterisation does not hold for infinite graphs: Oxley~\cite{oxley1979packing} gave countable locally finite graphs satisfying the characterisation but not containing $k$ edge-disjoint spanning trees.
However, Tutte proved that the characterisation is still valid for countable graphs if one asks for the edge-disjoint subgraphs to be \emph{semiconnected}: a subgraph $H \subset G$ is semiconnected if it contains an edge of every finite cut of $G$. 
We give an extension of this form of Theorem~\ref{thm:2treesgeneral} to countably infinite graphs.

\begin{theorem}\label{cor:infinite}
Let $G$ be a countably infinite graph containing a spanning double tree. 
The edges of $G$ may be coloured blue and red such that the blue and red graphs are both semiconnected and the blue and red degrees of each vertex differ by at most four \textnormal{(}or are both infinite\textnormal{)}.
Further, if $G$ is a double tree, the blue and red graphs may be chosen to be acyclic.
\end{theorem}

H\"orsch asked whether digraphs can be balanced with the role of trees being played by arborescences: rooted trees where all edges are directed away from the root.

\begin{question}\label{qu:digraphs}
Is there a constant $C$ such that the following holds for every digraph $D$ that is the union of two arc-disjoint arborescences? The edges of $D$ can be coloured blue and red such that the blue and red graphs are both arborescences and the blue and red out-degrees of each vertex differ by at most \(C\)?
\end{question}
We provide an infinite family of counterexamples that need $C \geqslant |V(D)| - 2$, answering H\"orsch's question in the negative.
We further show that the natural analogue where the role of trees is played by strongly connected digraphs is also false.

H\"orsch also asked an algorithmic question.
\begin{question}\label{conj:algo}
Does there exist a polynomial time algorithm to decide if a given Eulerian double tree has a perfectly balanced double tree decomposition?
\end{question}
We answer in the negative by reducing the NP-complete problem (\textit{cf.} P\'eroche~\cite{peroche1984npcompleteness}) of finding two edge-disjoint Hamiltonian cycles in a $4$-regular graph to this decision problem.

The paper is organised as follows.
In Section~\ref{sec:main} we give the proof of Theorem~\ref{thm:2trees} and the main tools used to prove Theorem~\ref{thm:2treesgeneral};
in Section~\ref{sec:general} we conclude the proof of Theorem~\ref{thm:2treesgeneral};
in Section~\ref{sec:semiconn} we discuss the infinite case and prove Theorem~\ref{cor:infinite};
in Section~\ref{sec:digraphs} we describe our constructions for digraphs;
in Section~\ref{sec:NP} we answer Question~\ref{conj:algo}. Finally, in Section~\ref{sec:conc}, we conclude with some natural questions of our own.

\subsection{Notation}
We use standard notation.
Throughout we consider all graphs to be multigraphs without self-loops.
For a graph $G = (V,E)$ with vertices $V$ and edges $E$ we write $e(G)$ for $|E|$ and $|G|$ for $|V|$.
Given a partition $A\sqcup B=V$ of the vertices of \(G\), we write $E(A,B)$ for the set of edges in $E$ with endpoints in both $A$ and $B$, and $e(A,B)= |E(A,B)|$.
If $A\subset V$ we write $G[A]$ to denote the subgraph induced by $A$.
When $X$ is a set of vertices (respectively edges) and $x$ is a vertex (respectively edge), we use the shorthand $X + x$ and $X - x$ to mean $X\cup \{x\}$ and $X\setminus \{x\}$.
For a graph $G$, a vertex $v$ and an edge $e$ we write $G-v$ for the graph obtained by deleting $v$ and all edges incident with it and $G-e$ for the graph obtained by deleting $e$. If $e \notin E(G)$, $G+e$ denotes the graph obtained by adding $e$ to $E(G)$.

We use \(\sqcup\) to denote a disjoint union.
Throughout we write $G = S_1\sqcup\dotsb\sqcup S_k$ to mean that $S_1,\dotsc,S_k$ are spanning subgraphs of $G$ and $E(G)= E(S_1)\sqcup \dotsb\sqcup E(S_k)$.
We sometimes refer to this as a \emph{decomposition} of $G$.
If $S_1,\dotsc,S_k$ are trees, we refer to it as a \emph{$k$-tree decomposition}; if further $k=2$, a \emph{double tree decomposition}.
\begin{definition*}
Let $G$ be a graph, $c$ an integer, and suppose $G = S_1 \sqcup\dotsb \sqcup S_k$.
A vertex $v\in V(G)$ is said to be \emph{$c$-balanced} in $S_1\sqcup \dotsb\sqcup S_k$ if for all $i,j$,
\begin{align*} 
|d_{S_i} (v) - d_{S_j} (v) | \leq c.
\end{align*}
We say that the decomposition $G = S_1 \sqcup \dotsb \sqcup S_k$ is \emph{$c$-balanced} if every $v\in V(G)$ is $c$-balanced in it. 
\end{definition*}
When the constant $c$ is clear, we write \emph{balanced} for brevity. Note, for example, that Theorem~\ref{thm:2trees} can be phrased as `every finite double tree admits a $4$-balanced double tree decomposition'.

\section{Balancing double trees}\label{sec:main}

Fix an integer $c\geq 2$ and suppose there are double trees with no \(c\)-balanced double tree decomposition. Throughout this section, we take \(G\) to be such a double tree with \(\lvert G\rvert\) minimal.

Call a vertex $v\in V(G)$ \textit{small} if $d_G(v) \leq c+2$ and \textit{big} otherwise.
A simple observation that will be used throughout the section is that small vertices are balanced in any double tree decomposition.
We call a vertex $v\in V(G)$ an \mbox{\textit{$\ell$-vertex}} if $d_G(v) = \ell$.

In Sections \ref{subsec:2verts}-\ref{subsec:crit} we show that a minimal counterexample $G$ satisfies a collection of structural properties.
Finally, in Section~\ref{subsec:discharge} we make a discharging argument with $c=4$ to conclude that $G$ has too many edges for a double tree.

Our arguments to show that $G$ cannot contain certain substructures have the following template.
\begin{enumerate}[noitemsep]
\item Locally modify $G$ to create a double tree $H$ with $|H|<|G|$.
\item Use minimality to find a balanced decomposition for $H$.
\item From this decomposition recover a decomposition for $G$.
\item Argue that this decomposition is balanced.
\end{enumerate}
Step 1 is referred to as the \emph{reduction step}, step 3 as the \emph{reconstruction step}.
In step 4, we need only show that the vertices involved in the reduction and reconstruction steps are balanced, as all other vertices are balanced in step 2 and left untouched afterwards.

Our methods refine those of H\"{o}rsch~\cite{horsch2021globally}. 
There are two main novel ideas. The first is using edge swaps to control the structure around certain 3-vertices (\emph{cf.}~Lemma~\ref{fact:poor3}). This is used to force blue and red degrees above 1 and so aid balancedness.
The second is controlling the parity of the degrees of the neighbours of 2-vertices.
These ideas are crucial to most of our structural lemmas.

\textit{In figures, big vertices are black, small vertices are white. 
When the status of a vertex is unclear, we indicate it in grey.
As a convention, when a graph has a double tree decomposition with trees labelled by 1 and 2, we use blue for tree 1, red for tree 2 and black when the colour is irrelevant.}

\subsection{\texorpdfstring{$2$}{2}-vertices}\label{subsec:2verts}
Let $v$ be a $2$-vertex and $x,y$ its (not necessarily distinct) neighbours. 
As $v$ is a leaf in both trees of any double tree decomposition of $G$, removing it yields a double tree $H$, which admits a balanced decomposition by minimality of $G$.
We refer to this as the \emph{standard reduction for $2$-vertices}.
This reduction can be reversed in the obvious way by adding back $v$ and the edges incident to it.
\begin{figure}[ht]
	\centering
	\begin{subfigure}{0.4\textwidth}
    \centering
	\begin{tikzpicture}
	[scale=.5,auto=center,circ/.style={circle,draw, minimum size=1pt, inner sep=2.5pt, fill=gray}, >=stealth]
	\begin{scope}
    	\node[circ, label={180:\small $x$}] (v1) at (-1.5,0) {};
    	\node[circ, label={360:\small $y$}] (v2) at (1.5,0) {};
    	\node[circ, fill=white, label={\small $v$}] (w) at (0,2) {};
    	\begin{scope}[every edge/.style={very thick,draw=red}]
    	\draw (v2) edge (w);
    	\end{scope}
    	\begin{scope}[every edge/.style={very thick,draw=blue}]
    	\draw (v1) edge (w);
    	\end{scope}
	\end{scope}
	\end{tikzpicture}
	\caption{Start configuration.} 
	\label{fig:2}
	\end{subfigure}
	\begin{subfigure}{0.4\textwidth}
    \centering
	\begin{tikzpicture}
	[scale=.5,auto=center,circ/.style={circle,draw, minimum size=1pt, inner sep=2.5pt, fill=gray}, >=stealth]
	\begin{scope}
    	\node[circ, label={180:\small $x$}] (v1) at (-1.5,0) {};
    	\node[circ, label={360:\small $y$}] (v2) at (1.5,0) {};
    	\node[minimum size=1pt, inner sep=2.5pt, label={ }, fill=white] (w) at (0,2) {};
	\end{scope}
	\end{tikzpicture}
	\caption{Reduced configuration.} 
	\end{subfigure}
	\caption{Standard reduction for $2$-vertices.}
\end{figure}
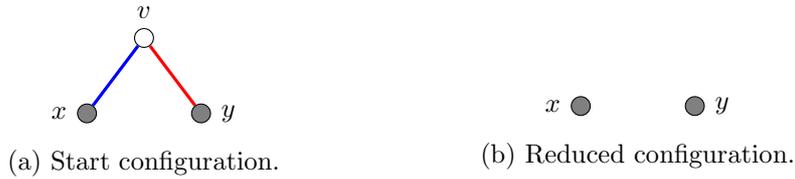
\begin{lemma}\label{lem:2verts}
Let $v\in V(G)$ be a $2$-vertex. 
Then the neighbours $x,y$ of $v$ are distinct, big and $d_G(x) \equiv d_G(y) \equiv c + 1 \bmod{2}$.
\end{lemma}

\begin{proof}
Let \(vx, vy\) be the edges incident to \(v\). If \(x = y\), then the standard reduction for \mbox{$2$-vertices} immediately gives a balanced decomposition for $G$, a contradiction.

Suppose \(x\) is small. Apply the standard reduction for $2$-vertices to obtain a double tree $H$ with a balanced blue/red double tree decomposition.
We may then add $v$ back by giving $vy$ a colour in which the degree of $y$ was smallest, and $vx$ the other colour.
Since $x$ is small, this yields a balanced decomposition for $G$, a contradiction.

Suppose for contradiction that $d_G(x)\equiv c \bmod{2}$.
Apply the standard reduction for \mbox{$2$-vertices} to $v$ to obtain a double tree $H$ with balanced decomposition $H= S_1\sqcup S_2$.
In particular, $| d_{S_1}(x) - d_{S_2}(x) | \leq c$. 
By the congruence condition,
\begin{equation*}
|d_{S_1}(x) - d_{S_2}(x) | \equiv 
d_{S_1}(x) + d_{S_2}(x)  = d_G(x) -1 \equiv c-1 \bmod{2},
\end{equation*}
so in fact, $|d_{S_1}(x) - d_{S_2}(x)| \leq c-1$. By symmetry we may assume that \(d_{S_1}(y) \geq d_{S_2}(y)\). Put \(v\) back, adding \(vx\) to \(S_1\) and \(vy\) to \(S_2\). Then \(y\) is still balanced and the degree difference at \(x\) has increased by at most \(1\), so is at most \(c\). This is a contradiction.
\end{proof}

The following observation appeared in H\"orsch~\cite[Lemma 2]{horsch2021globally}.

\begin{lemma}\label{lem:atmostone2vert}
Every big vertex $v\in V(G)$ is adjacent to at most one $2$-vertex. 
\end{lemma}

\subsection{Edge swaps}
We remind the reader of a simple tool that will be used repeatedly in our arguments.
Given a double tree decomposition $G= T_1\sqcup T_2$ and an edge $e\in E(T_1)$ we may swap it with some $f\in E(T_2)$ such that 
$G= S_1\sqcup S_2$,
where $E(S_1) = E(T_1) - e +f$ and $E(S_2)= E(T_2) - f +e$,
is a double tree decomposition.
Indeed, $T_1$ splits into two components after removing $e$ and adding $e$ to $T_2$ creates a cycle $C$.
We may thus choose $f$ to be any edge of $C - e$ with an endpoint in each component.
We will refer to this as \textit{swapping}~$e$.
Note that if $e \in E(T_1)$ is incident to a leaf $x$ of $T_1$, then $f$ must also be incident to $x$.
In particular, after swapping $e$, $x$ is still a leaf of $T_1$.
\begin{lemma}\label{lem:edgeswap}
Let $G$ be a double tree with a blue/red decomposition and $xy$ a blue edge such that $x$ is a leaf in the blue tree.
Then $x$ remains a leaf in the blue tree after swapping $xy$.
\end{lemma}
This is particularly useful for $3$-vertices as in any blue/red double tree decomposition a $3$-vertex must be a leaf in some colour.

\subsection{\texorpdfstring{$3$}{3}-vertices}\label{sec:3verts}

Let $v$ be a $3$-vertex with (not necessarily distinct) neighbours $x,y,b$ where the edges \(vx, vy\) are red and the edge \(vb\) is blue as in Figure~\ref{fig:3red}.
Remove $v$ and join $x$ and $y$ in red to form $H$;
we will refer to this as the \emph{standard reduction for $3$-vertices}. If \(v\) has two blue neighbours and one red neighbour, then there is an analogous reduction.
Since $v$ was a leaf in the blue tree and $xvy$ is the only path from $x$ to $y$ in the red tree, the resulting blue/red decomposition is a double tree decomposition for $H$.
Further, $|H| = |G|-1$ so, by minimality, $H$ has a balanced double tree decomposition. A particularly useful feature of this reduction is that it is reversible: given a double tree containing the configuration shown in Figure~\ref{fig:3redb} we may delete the edge $xy$, add a new vertex $v$ joined to $x$ and $y$ in red and joined to $b$ in blue to form another double tree.

\begin{figure}[ht]
	\centering
	\begin{subfigure}{0.4\textwidth}
    \centering
	\begin{tikzpicture}
	[scale=.5,auto=center,circ/.style={circle,draw, minimum size=1pt, inner sep=2.5pt, fill=gray}, >=stealth]
	\begin{scope}
    	\node[circ, label={180:\small $x$}] (v1) at (0,0) {};
    	\node[circ, label={180:\small $y$}] (v2) at (0,-4) {};
    	\node[circ, label={0:\small $b$}] (v3) at (2.5,-2) {};
    	\node[circ, fill=white, label={180:\small $v$}] (w) at (0,-2) {};
    	\begin{scope}[every edge/.style={very thick,draw=red}]
    	\draw (v2) edge (w);
    	\draw (v1) edge (w);
    	\end{scope}
    	\begin{scope}[every edge/.style={very thick,draw=blue}]
    	\draw (v3) edge (w);
    	\end{scope}
	\end{scope}
	\end{tikzpicture}
	\caption{Start configuration.} 
	\label{fig:3red}
	\end{subfigure}
	\begin{subfigure}{0.4\textwidth}
    \centering
	\begin{tikzpicture}
	[scale=.5,auto=center,circ/.style={circle,draw, minimum size=1pt, inner sep=2.5pt, fill=gray}, >=stealth]
	\begin{scope}
    	\node[circ, label={180:\small $x$}] (v1) at (0,0) {};
    	\node[circ, label={180:\small $y$}] (v2) at (0,-4) {};
    	\node[circ, label={0:\small $b$}] (v3) at (2.5,-2) {};
    	\begin{scope}[every edge/.style={very thick,draw=red}]
    	\draw (v1) edge (v2);
    	\end{scope}
	\end{scope}
	\end{tikzpicture}
	\caption{Reduced configuration.}
	\label{fig:3redb}
	\end{subfigure}
	\caption{Standard reduction for $3$-vertices.}
\end{figure}
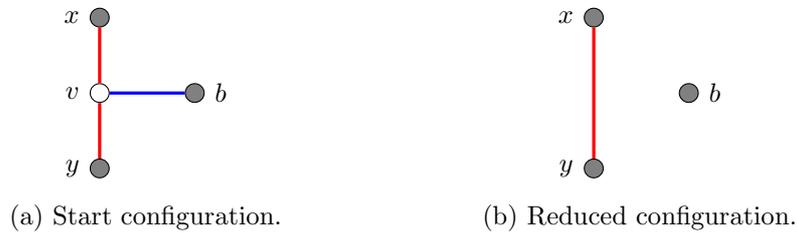
Using this reduction, H\"orsch~\cite[Proposition 10]{horsch2021globally} showed the following.
\begin{lemma}\label{lem:3vertsbig}
    Let \(v \in V(G)\) be a \(3\)-vertex and \(G = T_1 \sqcup T_2 \) a double tree decomposition. Suppose that \(v\) is a leaf in \(T_2\) with \(vb \in E(T_2)\) its unique incident edge. Then \(b\) is big.
\end{lemma}
Suppose that \(v\) is a \(3\)-vertex with two edges to small vertices. Let \(vb\) be the third edge incident to \(v\). Lemma~\ref{lem:3vertsbig} shows that in any double tree decomposition of \(G\), the edge \(vb\) is of the opposite colour to the other two edges. However, swapping the edge \(vb\) gives a contradiction. So every \(3\)-vertex has at most one edge to a small vertex and so only the following types of 3-vertices can occur.

\begin{definition}[types of \(3\)-vertex]
We say that a $3$-vertex is
\begin{itemize}[noitemsep, topsep=0pt]
\item \emph{rich} if all its neighbours are big;
\item \emph{poor} if it is adjacent to three distinct vertices, two big, one small;
\item \emph{bad} if it has a small neighbour and is joined to a big vertex by a double edge.
\end{itemize}
\end{definition}
We are ready to prove a key result that gives structure around poor \(3\)-vertices.

\begin{lemma}\label{fact:poor3}
Let $v\in V(G)$ be a poor $3$-vertex with big neighbours $x,y$ and small neighbour $s$.
In any tree decomposition $G=T_1\sqcup T_2$ where \(vs \in E(T_1)\)\textnormal{:}
\begin{itemize}[noitemsep]
    \item Exactly one of \(vx, vy\) is in \(E(T_1)\).
    \item If \(vy \in E(T_2)\), then swapping \(vy\) gives the double tree decomposition \(G = T'_1 \sqcup T'_2\), where \(E(T'_1) = E(T_1) - vx + vy\) and \(E(T'_2) = E(T_2) + vx - vy\).
    \item The path from \(x\) to \(y\) in \(T_1\) does not contain \(s\).
\end{itemize}
\end{lemma}
\begin{proof}
\begin{figure}[ht] 
	\centering
	\begin{subfigure}{0.4\textwidth}
    \centering
	\begin{tikzpicture}
	[scale=.5,auto=center,circ/.style={circle,draw, minimum size=1pt, inner sep=2.5pt, fill=black}, >=stealth]
	\tikzset{snake it/.style={decorate, decoration=snake}}
	\begin{scope}
    	\node[circ, label={180:\small $x$}] (v1) at (0,0) {};
    	\node[circ, label={180:\small $y$}] (v2) at (0,-4) {};
    	\node[circ, fill=white, label={0:\small $s$}] (v3) at (2.5,-2) {};
    	\node[circ, fill=white, label={180:\small $v$}] (w) at (0,-2) {};
    	\begin{scope}[every edge/.style={very thick,draw=red}]
    	\draw (v2) edge (w);
    	\end{scope}
    	\begin{scope}[every edge/.style={very thick,draw=blue}]
    	\draw (v3) edge (w);
    	\draw (v1) edge (w);
    	        segment amplitude=.6mm,
	            segment length=3mm,
	            line after snake=0mm] (v2);
    	\end{scope}
	\end{scope}
	\end{tikzpicture}
	\caption{\(v\) a leaf in \(T_2\) with \(vy \in E(T_2)\).}\label{fig:Fact2a} 
	\end{subfigure}\hspace{1em}
	\begin{subfigure}{0.4\textwidth}
    \centering
	\begin{tikzpicture}
	[scale=.5,auto=center,circ/.style={circle,draw, minimum size=1pt, inner sep=2.5pt, fill=black}, >=stealth]
	\begin{scope}
    	\node[circ, label={180:\small $x$}] (x) at (0,0) {};
    	\node[circ, label={180:\small $y$}] (y) at (0,-4) {};
    	\node[circ, fill=white, label={0:\small $s$}] (s) at (2.5,-2) {};
    	\node[circ, fill=white, label={180:\small $v$}] (v) at (0,-2) {};
    	\begin{scope}[every edge/.style={very thick,draw=red}]
    	\draw (x) edge (v);
    	\end{scope}
    	\begin{scope}[every edge/.style={very thick,draw=blue}]
    	\draw (y) edge (v);
    	\draw (s) edge (v);
    	\end{scope}
	\end{scope}
	\end{tikzpicture}
	\caption{Configuration after swapping $vy$.}
	\end{subfigure}
	\caption{Edge swap in Lemma~\ref{fact:poor3}.}
\end{figure}
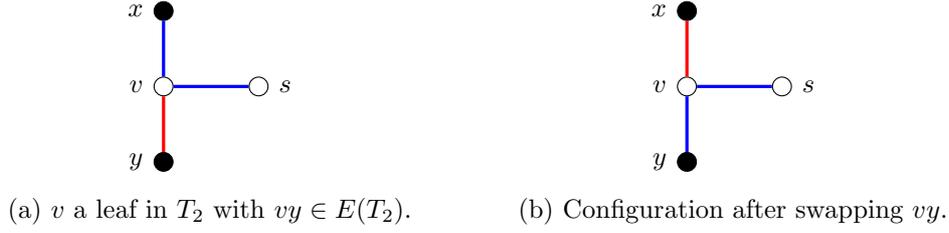
Firstly, by Lemma~\ref{lem:3vertsbig}, at least one of \(vx, vy\) is in \(E(T_1)\). They cannot both be otherwise \(v\) is an isolated vertex in \(T_2\). This gives the first bullet point.

Suppose that \(vy \in E(T_2)\) and so \(vx \in E(T_1)\) and \(v\) is a leaf in \(T_2\) as in Figure~\ref{fig:Fact2a}. Consider swapping \(vy\). Then \(vy\) becomes blue and so, by the first bullet point, \(vx\) becomes red.

We now prove the third bullet point. We may assume by symmetry that \(vy \in E(T_2)\) and so \(vx \in E(T_1)\). Suppose that there is a path \(P\) in \(T_1\) from \(x\) to \(y\) that contains \(s\) and let $P'$ be the subpath of $P$ from $s$ to $y$. Consider swapping $vy$: we have just shown that $vx$ becomes red. But then $yvsP'$ forms a blue cycle, which is impossible.
\end{proof}
\begin{lemma}\label{lem:badvertices}
Let $v\in V(G)$ be adjacent to $\ell\geq 1$ bad $3$-vertices via their double edges.
Then,
\begin{equation*}
d_G(v) \geq 2\ell +c+1.
\end{equation*}
\end{lemma}
\begin{proof}
Let $v$ be a vertex in $G$ with a bad neighbour $u$, and let $w\neq v$ be the small neighbour of $u$.
Fix a double tree decomposition $G=T_1\sqcup T_2$.
By symmetry we may assume $uw\in E(T_1)$.

Apply the standard reduction for $3$-vertices to $u$ and let $H$ be the resulting graph with balanced double tree decomposition $H=S_1\sqcup S_2$.
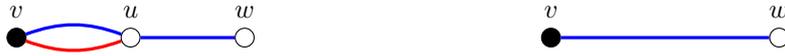
\begin{figure}[ht]
	\centering
	\begin{tikzpicture}
	[scale=.5,auto=center,circ/.style={circle,draw, minimum size=1pt, inner sep=2.5pt, fill=black}, >=stealth]
	\begin{scope}
    	\node[circ, label={\small $v$}] (v) at (0,0) {};
    	\node[circ, fill=white, label={\small $u$}] (u) at (3,0) {};
    	\node[circ, fill=white, label={\small $w$}] (w) at (6,0) {};
    	\begin{scope}[every edge/.style={very thick,draw=red}]
    	\draw [bend right=20] (v) edge (u);
    	\end{scope}
    	\begin{scope}[every edge/.style={very thick,draw=blue}]
    	\draw [bend left=20] (v) edge (u);
    	\draw (u) edge (w);
    	\end{scope}
	\end{scope}
	\begin{scope}[xshift=400pt]
    	\node[circ, label={\small $v$}] (v) at (0,0) {};
    	\node[circ, fill=white, label={\small $w$}] (w) at (6,0) {};
    	\begin{scope}[every edge/.style={very thick,draw=blue}]
    	\draw (v) edge (w);
    	\end{scope}
	\end{scope}
	\end{tikzpicture}
	\caption{Reduction step in Lemma~\ref{lem:badvertices}.}
\end{figure}
Without loss of generality we may assume that $vw\in E(S_1)$.
Define $G= T_1'\sqcup T_2'$, where
\begin{align*}
E(T_1') &= E(S_1) - vw + vu + uw,\\
E(T_2') &= E(S_2) + vu,
\end{align*}
reversing the reduction. All vertices, except possibly $v$, are balanced in $T'_1 \sqcup T'_2$.
Since $v$ is adjacent to $\ell$ bad $3$-vertices via double edges, we have $d_{T'_1}(v) \geq \ell$ and $d_{T'_2}(v)\geq \ell$. In particular, if $d_G(v) \leq 2\ell + c$, then $v$ is balanced also, which is a contradiction.
\end{proof}
\subsection{Critical vertices}\label{subsec:crit}
A vertex $v$ in $G$ is said to be \emph{critical} if $d_G(v)=c+3$, that is, if its degree is just large enough for it to be big.
A simple observation is that critical vertices are balanced in a blue/red decomposition if and only if both their blue and red degrees are at least two.
We combine this observation with the final bullet point of Lemma~\ref{fact:poor3} to great effect: suppose a vertex $x$ has a blue edge to a poor \(3\)-vertex $v$ and $v$ has blue degree two in a given blue/red decomposition. Then the final bullet point of Lemma~\ref{fact:poor3} guarantees that \(v\) has blue degree at least two.
If further $v$ is critical and has red degree at least two, then it is balanced.
\begin{lemma}\label{lem:7verts} 
Let $v\in V(G)$ be a critical vertex.
\begin{enumerate}[noitemsep, label=\emph{(}\roman*\emph{)}]
\item If all neighbours of $v$ are small, then $v$ is not adjacent to bad $3$-vertices.
\item At most one neighbour of $v$ is a poor $3$-vertex.
\item If $v$ is adjacent to a $2$-vertex, then \(v\) is not adjacent to a poor $3$-vertex.
\end{enumerate}
\end{lemma}
Note that Lemmas~\ref{lem:atmostone2vert}, \ref{lem:7verts}.(ii) and \ref{lem:7verts}.(iii) yield that a critical vertex has at most one neighbour that is either a 2-vertex or a poor 3-vertex.

\begin{proof} \textbf{(i)}
Suppose not and let $G = T_1\sqcup T_2$ be a double tree decomposition. 
Let $v\in V(G)$ be critical with all neighbours small and let $u\in \Gamma(v)$ be a bad $3$-vertex, and $w\neq v$ be the small neighbour of $u$.
By symmetry we may assume $uw\in E(T_1)$.

Apply the standard reduction for $3$-vertices to $u$, let $H$ be the resulting double tree and $H= S_1\sqcup S_2$ a balanced double tree decomposition.
Without loss of generality we may assume that $vw\in E(S_1)$.

\textit{Case 1.} $d_{S_1}(v) \geq 2$.

Reverse the reduction to get $G=T_1'\sqcup T_2'$, where 
\begin{align*}
E(T_1') &= E(S_1) - vw + vu + uw,\\
E(T_2') &= E(S_2) + vu.
\end{align*}

Then $d_{T_1'}(v) = d_{S_1}(v)\geq 2$ and \(d_{T'_2}(v) \geq d_{S_2}(v) + 1 \geq 2\), and so, since $v$ is critical, it is balanced and thus the decomposition $T_1'\sqcup T_2'$ is balanced, a
contradiction.

\textit{Case 2.} $d_{S_1}(v) =1$.

Let $H=S_1'\sqcup S_2'$ be the decomposition obtained after swapping $vw$.
Since $v$ is a leaf in $S_1$, it remains a leaf in $S_1'$.
Moreover, every vertex in the neighbourhood of $v$ is small, so every big vertex is balanced in $S_1'\sqcup S_2'$.
Now $vw\in E(S_2')$ and $d_{S_2'}(v) \geq 2$, so by Case 1 we can find a balanced decomposition $G= T_1'\sqcup T_2'$, a 
contradiction.

\textbf{(ii)} Suppose that $v\in V(G)$ is critical and $u,w\in \Gamma(v)$ are distinct poor $3$-vertices.
Let the other neighbours of $u$ and $w$ be $u_1,u_2$ and $w_1,w_2$ (not necessarily distinct) respectively, as in Figure~\ref{fig:critred1a}, where $u_2$, $w_2$ are small.
By Lemma~\ref{fact:poor3} we may perform edge swaps to ensure $\{vw, ww_2\}\subset E(T_i)$ and $\{uv,uu_2\} \subset E(T_j)$ for some $i,j\in \{1,2\}$.
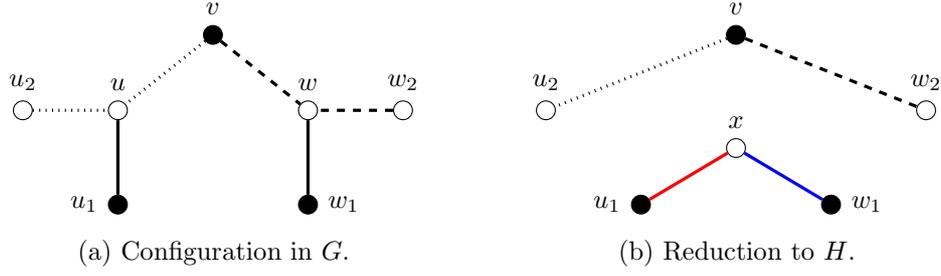
\begin{figure}[ht]
	\centering
	\begin{subfigure}{0.4\textwidth}
    \centering
	\begin{tikzpicture}
	[scale=.5,auto=center,circ/.style={circle,draw, minimum size=1pt, inner sep=2.5pt, fill=black}, >=stealth]
	\begin{scope}
    	\node[circ, label={\small $v$}] (v) at (0,0) {};
    	\node[circ, fill=white, label={\small $u_2$}] (u2) at (-5,-2) {};
    	\node[circ, label={180:\small $u_1$}] (u1) at (-2.5,-4.5) {};
    	\node[circ, label={0:\small $w_1$}] (w1) at (2.5,-4.5) {};
    	\node[circ, fill=white, label={\small $w_2$}] (w2) at (5,-2) {};
    	\node[circ, fill=white, label={\small $u$}] (u) at (-2.5,-2) {};
    	\node[circ, fill=white, label={\small $w$}] (w) at (2.5,-2) {};
    	\begin{scope}[every edge/.style={very thick,draw=black}]
    	\draw [dotted] (v) edge (u);
    	\draw [dotted] (u2) edge (u);
    	\draw (u1) edge (u);
    	\draw [dashed] (v) edge (w);
    	\draw (w1) edge (w);
    	\draw [dashed] (w2) edge (w);
    	\end{scope}
	\end{scope}
	\end{tikzpicture}
	\caption{Configuration in $G$.}
	\label{fig:critred1a}
	\end{subfigure}\hspace{1em}
	\begin{subfigure}{0.4\textwidth}
    \centering
	\begin{tikzpicture}
	[scale=.5,auto=center,circ/.style={circle,draw, minimum size=1pt, inner sep=2.5pt, fill=black}, >=stealth]
	\begin{scope}
    	\node[circ, label={\small $v$}] (v) at (0,0) {};
    	\node[circ, fill=white, label={\small $u_2$}] (u2) at (-5,-2) {};
    	\node[circ, label={180:\small $u_1$}] (u1) at (-2.5,-4.5) {};
    	\node[circ, label={0:\small $w_1$}] (w1) at (2.5,-4.5) {};
    	\node[circ, fill=white, label={\small $w_2$}] (w2) at (5,-2) {};
    	\node[circ, fill=white, label={\small $x$}] (x) at (0,-3) {};
    	\begin{scope}[every edge/.style={very thick,draw=red}]
    	\draw (x) edge (u1);
    	\end{scope}
    	\begin{scope}[every edge/.style={very thick,draw=blue}]
    	\draw (x) edge (w1);
    	\end{scope}
    	\begin{scope}[every edge/.style={very thick,draw=black}]
    	\draw [dotted] (v) edge (u2);
    	\draw [dashed] (v) edge (w2);
    	\end{scope}
	\end{scope}
	\end{tikzpicture}
	\caption{Reduction to $H$.}
	\end{subfigure}
	\caption{Reduction step in Lemma~\ref{lem:7verts}.(ii). Dashed edges are in the same tree, dotted edges are in the same tree.}
\end{figure}
Apply the standard reduction for $3$-vertices to $u$ and $w$ and add a $2$-vertex $x$ adjacent to both $u_1$ and $w_1$, yielding a double tree $H$ which by induction has a balanced double tree decomposition $H = S_1\sqcup S_2$.
Without loss of generality we may assume that $vw_2\in E(S_1)$.
We consider multiple cases.

\textit{Case 1.}
$vu_2\in E(S_1)$.

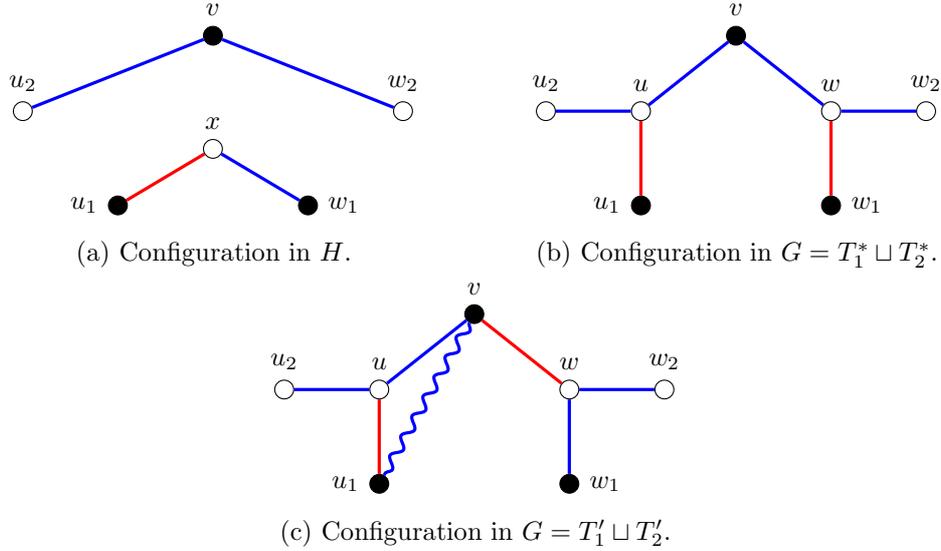
\begin{figure}[ht]
	\centering
	\begin{subfigure}{0.4\textwidth}
    \centering
	\begin{tikzpicture}
	[scale=.5,auto=center,circ/.style={circle,draw, minimum size=1pt, inner sep=2.5pt, fill=black}, >=stealth]
    	\node[circ, label={\small $v$}] (v) at (0,0) {};
    	\node[circ, fill=white, label={\small $u_2$}] (u2) at (-5,-2) {};
    	\node[circ, label={180:\small $u_1$}] (u1) at (-2.5,-4.5) {};
    	\node[circ, label={0:\small $w_1$}] (w1) at (2.5,-4.5) {};
    	\node[circ, fill=white, label={\small $w_2$}] (w2) at (5,-2) {};
    	\node[circ, fill=white, label={\small $x$}] (x) at (0,-3) {};
    	\begin{scope}[every edge/.style={very thick,draw=red}]
    	\draw (x) edge (u1);
    	\end{scope}
    	\begin{scope}[every edge/.style={very thick,draw=blue}]
    	\draw (x) edge (w1);
    	\draw (v) edge (u2);
    	\draw (v) edge (w2);
    	\end{scope}
    	\begin{scope}[every edge/.style={very thick,draw=black}]
    	\end{scope}
	\end{tikzpicture}
	\caption{Configuration in $H$.}
	\end{subfigure}\hspace{1em}
	\begin{subfigure}{0.4\textwidth}
    \centering
	\begin{tikzpicture}
	[scale=.5,auto=center,circ/.style={circle,draw, minimum size=1pt, inner sep=2.5pt, fill=black}, >=stealth]
	\begin{scope}
    	\node[circ, label={\small $v$}] (v) at (0,0) {};
    	\node[circ, fill=white, label={\small $u_2$}] (u2) at (-5,-2) {};
    	\node[circ, label={180:\small $u_1$}] (u1) at (-2.5,-4.5) {};
    	\node[circ, label={0:\small $w_1$}] (w1) at (2.5,-4.5) {};
    	\node[circ, fill=white, label={\small $w_2$}] (w2) at (5,-2) {};
    	\node[circ, fill=white, label={\small $u$}] (u) at (-2.5,-2) {};
    	\node[circ, fill=white, label={\small $w$}] (w) at (2.5,-2) {};
    	\begin{scope}[every edge/.style={very thick,draw=red}]
    	\draw (u1) edge (u);
    	\draw (w1) edge (w);
    	\end{scope}
    	\begin{scope}[every edge/.style={very thick,draw=blue}]
    	\draw (v) edge (u);
    	\draw (u2) edge (u);
    	\draw (v) edge (w);
    	\draw (w2) edge (w);
    	\end{scope}
    	\begin{scope}[every edge/.style={very thick,draw=black}]
    	\end{scope}
	\end{scope}
	\end{tikzpicture}
	\caption{Configuration in $G=T_1^{\ast}\sqcup T_2^{\ast}$.}
	\label{fig:critrec1b}
	\end{subfigure}
	\begin{subfigure}{0.4\textwidth}
    \centering
	\begin{tikzpicture}
	[scale=.5,auto=center,circ/.style={circle,draw, minimum size=1pt, inner sep=2.5pt, fill=black}, >=stealth]
    	\node[circ, label={\small $v$}] (v) at (0,0) {};
    	\node[circ, fill=white, label={\small $u_2$}] (u2) at (-5,-2) {};
    	\node[circ, label={180:\small $u_1$}] (u1) at (-2.5,-4.5) {};
    	\node[circ, label={0:\small $w_1$}] (w1) at (2.5,-4.5) {};
    	\node[circ, fill=white, label={\small $w_2$}] (w2) at (5,-2) {};
    	\node[circ, fill=white, label={\small $u$}] (u) at (-2.5,-2) {};
    	\node[circ, fill=white, label={\small $w$}] (w) at (2.5,-2) {};
    	\begin{scope}[every edge/.style={very thick,draw=red}]
    	\draw (u1) edge (u);
    	\draw (v) edge (w);
    	\end{scope}
    	\begin{scope}[every edge/.style={very thick,draw=blue}]
    	\draw (w1) edge (w);
    	\draw (v) edge (u);
    	\draw (u2) edge (u);
    	\draw (w2) edge (w);
    	\draw (v) edge[ snake=snake,
    	        segment amplitude=.6mm,
	            segment length=3mm,
	            line after snake=0mm] (u1);
    	\end{scope}
    	\begin{scope}[every edge/.style={very thick,draw=black}]
    	\end{scope}
	\end{tikzpicture}
	\caption{Configuration in $G=T_1'\sqcup T_2'$.}
	\label{fig:critrec1c}
	\end{subfigure}
	\caption{Reconstruction in Case 1.}
\end{figure}

By symmetry we may assume that $xw_1\in E(S_1)$.
Reverse the reductions and delete $x$ to give $G= T_1^{\ast} \sqcup T_2^{\ast}$, where 
\begin{align*}
E(T_1^{\ast}) &= E(S_1) - xw_1 -vu_2 - vw_2 + vu + uu_2 + vw+ww_2,\\
E(T_2^{\ast}) &= E(S_2) -xu_1 + uu_1 + ww_1,
\end{align*}
as in Figure~\ref{fig:critrec1b}.
Consider swapping $ww_1$. 
Lemma~\ref{fact:poor3} implies that we get the decomposition $G= T_1' \sqcup T_2'$ shown in Figure~\ref{fig:critrec1c}, where
\begin{align*}
E(T_1') &= E(T_1^{\ast}) -vw + ww_1,\\
E(T_2') &= E(T_2^{\ast}) -ww_1 + vw.
\end{align*}

We claim that $G =T_1'\sqcup T_2'$ is balanced.
All degree differences are the same as in \(S_1 \sqcup S_2\) except at \(v\) where a blue edge has become red.
By Lemma~\ref{fact:poor3}, there is a path in $T_1'$ from $v$ to $u_1$ that does not use $u_2$. Since also $uv \in E(T_1')$, we get $d_{T_1'}(v) \geq 2$.
As $v$ is critical this means it is balanced, and therefore $G=T_1'\sqcup T_2'$ is balanced, as required.

\textit{Case 2.i.}
$vu_2\in S_2$ and $xu_1\in S_1$.

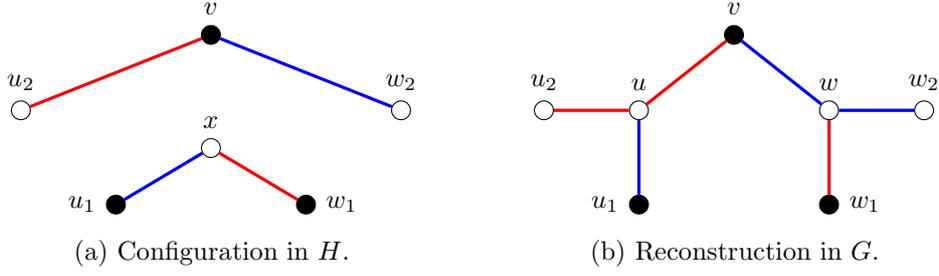
\begin{figure}[ht]
	\centering
	
	\begin{subfigure}{0.4\textwidth}
    \centering
	\begin{tikzpicture}
	[scale=.5,auto=center,circ/.style={circle,draw, minimum size=1pt, inner sep=2.5pt, fill=black}, >=stealth]
    	\node[circ, label={\small $v$}] (v) at (0,0) {};
    	\node[circ, fill=white, label={\small $u_2$}] (u2) at (-5,-2) {};
    	\node[circ, label={180:\small $u_1$}] (u1) at (-2.5,-4.5) {};
    	\node[circ, label={0:\small $w_1$}] (w1) at (2.5,-4.5) {};
    	\node[circ, fill=white, label={\small $w_2$}] (w2) at (5,-2) {};
    	\node[circ, fill=white, label={\small $x$}] (x) at (0,-3) {};
    	\begin{scope}[every edge/.style={very thick,draw=red}]
    	\draw (v) edge (u2);
    	\draw (x) edge (w1);
    	\end{scope}
    	\begin{scope}[every edge/.style={very thick,draw=blue}]
    	\draw (x) edge (u1);
    	\draw (v) edge (w2);
    	\end{scope}
    	\begin{scope}[every edge/.style={very thick,draw=black}]
    	\end{scope}
	\end{tikzpicture}
	\caption{Configuration in $H$.}
	\end{subfigure}\hspace{1em}
	\begin{subfigure}{0.4\textwidth}
    \centering
	\begin{tikzpicture}
	[scale=.5,auto=center,circ/.style={circle,draw, minimum size=1pt, inner sep=2.5pt, fill=black}, >=stealth]
    	\node[circ, label={\small $v$}] (v) at (0,0) {};
    	\node[circ, fill=white, label={\small $u_2$}] (u2) at (-5,-2) {};
    	\node[circ, label={180:\small $u_1$}] (u1) at (-2.5,-4.5) {};
    	\node[circ, label={0:\small $w_1$}] (w1) at (2.5,-4.5) {};
    	\node[circ, fill=white, label={\small $w_2$}] (w2) at (5,-2) {};
    	\node[circ, fill=white, label={\small $u$}] (u) at (-2.5,-2) {};
    	\node[circ, fill=white, label={\small $w$}] (w) at (2.5,-2) {};
    	\begin{scope}[every edge/.style={very thick,draw=red}]
    	\draw (w1) edge (w);
    	\draw (v) edge (u);
    	\draw (u2) edge (u);
    	\end{scope}
    	\begin{scope}[every edge/.style={very thick,draw=blue}]
    	\draw (u1) edge (u);
    	\draw (v) edge (w);
    	\draw (w2) edge (w);
    	\end{scope}
    	\begin{scope}[every edge/.style={very thick,draw=black}]
    	\end{scope}
    \end{tikzpicture}
	\caption{Reconstruction in $G$.}
	\label{fig:critrec2ib}
	\end{subfigure}
	\caption{Reconstruction in Case 2.i.}
\end{figure}

Reverse the reductions to get $G= T_1' \sqcup T_2'$, where 
\begin{align*}
E(T_1') &= E(S_1) -xu_1 -vw_2 + uu_1 + vw+ww_2,\\
E(T_2') &= E(S_2) - xw_1 - vu_2 + ww_1 + uv + uu_2,
\end{align*}
as in Figure~\ref{fig:critrec2ib}.
Further, since $S_1\sqcup S_2$ is balanced and degree differences of big vertices remained unchanged, $T_1'\sqcup T_2'$ is balanced, a
contradiction.

\textit{Case 2.ii.}
$vu_2\in E(S_2)$ and $xu_1\in E(S_2)$.

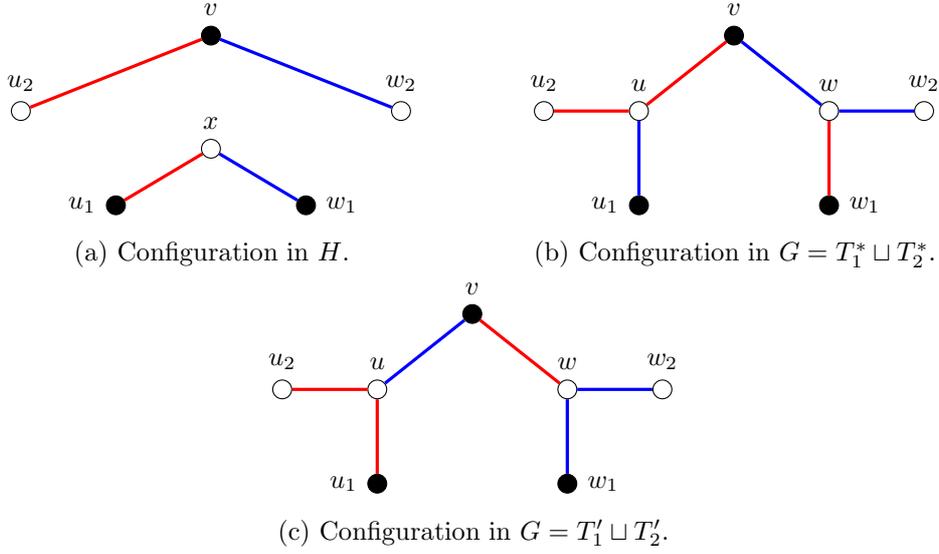
\begin{figure}[ht]
	\centering
	\begin{subfigure}{0.4\textwidth}
    \centering
	\begin{tikzpicture}
	[scale=.5,auto=center,circ/.style={circle,draw, minimum size=1pt, inner sep=2.5pt, fill=black}, >=stealth]
    	\node[circ, label={\small $v$}] (v) at (0,0) {};
    	\node[circ, fill=white, label={\small $u_2$}] (u2) at (-5,-2) {};
    	\node[circ, label={180:\small $u_1$}] (u1) at (-2.5,-4.5) {};
    	\node[circ, label={0:\small $w_1$}] (w1) at (2.5,-4.5) {};
    	\node[circ, fill=white, label={\small $w_2$}] (w2) at (5,-2) {};
    	\node[circ, fill=white, label={\small $x$}] (x) at (0,-3) {};
    	\begin{scope}[every edge/.style={very thick,draw=red}]
    	\draw (v) edge (u2);
    	\draw (x) edge (u1);
    	\end{scope}
    	\begin{scope}[every edge/.style={very thick,draw=blue}]
    	\draw (x) edge (w1);
    	\draw (v) edge (w2);
    	\end{scope}
    	\begin{scope}[every edge/.style={very thick,draw=black}]
    	\end{scope}
	\end{tikzpicture}
	\caption{Configuration in $H$.}
	\end{subfigure}\hspace{1em}
	\begin{subfigure}{0.4\textwidth}
    \centering
	\begin{tikzpicture}
	[scale=.5,auto=center,circ/.style={circle,draw, minimum size=1pt, inner sep=2.5pt, fill=black}, >=stealth]
    	\node[circ, label={\small $v$}] (v) at (0,0) {};
    	\node[circ, fill=white, label={\small $u_2$}] (u2) at (-5,-2) {};
    	\node[circ, label={180:\small $u_1$}] (u1) at (-2.5,-4.5) {};
    	\node[circ, label={0:\small $w_1$}] (w1) at (2.5,-4.5) {};
    	\node[circ, fill=white, label={\small $w_2$}] (w2) at (5,-2) {};
    	\node[circ, fill=white, label={\small $u$}] (u) at (-2.5,-2) {};
    	\node[circ, fill=white, label={\small $w$}] (w) at (2.5,-2) {};
    	\begin{scope}[every edge/.style={very thick,draw=red}]
    	\draw (w1) edge (w);
    	\draw (v) edge (u);
    	\draw (u2) edge (u);
    	\end{scope}
    	\begin{scope}[every edge/.style={very thick,draw=blue}]
    	\draw (u1) edge (u);
    	\draw (v) edge (w);
    	\draw (w2) edge (w);
    	\end{scope}
    	\begin{scope}[every edge/.style={very thick,draw=black}]
    	\end{scope}
	\end{tikzpicture}
	\caption{Configuration in $G= T_1^{\ast}\sqcup T_2^{\ast}$.}
	\label{fig:critrec2iib}
	\end{subfigure}
	\begin{subfigure}{0.4\textwidth}
    \centering
	\begin{tikzpicture}
	[scale=.5,auto=center,circ/.style={circle,draw, minimum size=1pt, inner sep=2.5pt, fill=black}, >=stealth]
    	\node[circ, label={\small $v$}] (v) at (0,0) {};
    	\node[circ, fill=white, label={\small $u_2$}] (u2) at (-5,-2) {};
    	\node[circ, label={180:\small $u_1$}] (u1) at (-2.5,-4.5) {};
    	\node[circ, label={0:\small $w_1$}] (w1) at (2.5,-4.5) {};
    	\node[circ, fill=white, label={\small $w_2$}] (w2) at (5,-2) {};
    	\node[circ, fill=white, label={\small $u$}] (u) at (-2.5,-2) {};
    	\node[circ, fill=white, label={\small $w$}] (w) at (2.5,-2) {};
    	\begin{scope}[every edge/.style={very thick,draw=red}]
    	\draw (u2) edge (u);
    	\draw (v) edge (w);
    	\draw (u1) edge (u);
    	\end{scope}
    	\begin{scope}[every edge/.style={very thick,draw=blue}]
    	\draw (w1) edge (w);
    	\draw (v) edge (u);
    	\draw (w2) edge (w);
    	\end{scope}
    	\begin{scope}[every edge/.style={very thick,draw=black}]
    	\end{scope}
	\end{tikzpicture}
	\caption{Configuration in $G= T_1'\sqcup T_2'$.}
	\end{subfigure}
	\caption{Reconstruction in Case 2.ii.}
\end{figure}

Reverse the reductions to get $G= T_1^{\ast} \sqcup T_2^{\ast}$, where 
\begin{align*}
E(T_1^{\ast}) &= E(S_1) - xw_1 - vw_2 + uu_1 + vw +ww_2,\\
E(T_2^{\ast}) &= E(S_2) -xu_1 -vu_2+ vu + uu_2 + ww_1,
\end{align*}
as in Figure~\ref{fig:critrec2iib}.
By edge swapping $uu_1$ and $ww_1$ and noting that $u$, $w$ are leaves in $T_1^{\ast}$, $T_2^{\ast}$ respectively, we obtain the double tree decomposition $G= T_1'\sqcup T_2'$, where
\begin{align*}
E(T_1') &= E(T_1^{\ast}) - uu_1 - vw + uv + ww_1,\\
E(T_2') &= E(T_2^{\ast}) - uv - ww_1 + uu_1 + vw.
\end{align*}
Vertices $v$, $u_1$, $u_2$ are balanced in $S_1\sqcup S_2$, hence with respect to $T_1'\sqcup T_2'$ as well, as degree differences remained unchanged.
All other degree differences at big vertices were preserved, so $G=T_1'\sqcup T_2'$ is balanced, a contradiction.

\textbf{(iii)} Let $v\in V(G)$ be critical, $G=T_1\sqcup T_2$ be a double tree decomposition and suppose that $u,w\in \Gamma(v)$ are a $2$-vertex and a poor $3$-vertex, respectively.
Let $v_3$ be the small neighbour of $w$, and $v_1,v_2\neq v$ the other big neighbours of $u$, $w$ respectively.
By symmetry we may assume that $w$ is a leaf in $T_2$. By Lemma~\ref{fact:poor3}, we may swap edges to ensure that $wv_2\in E(T_2)$.
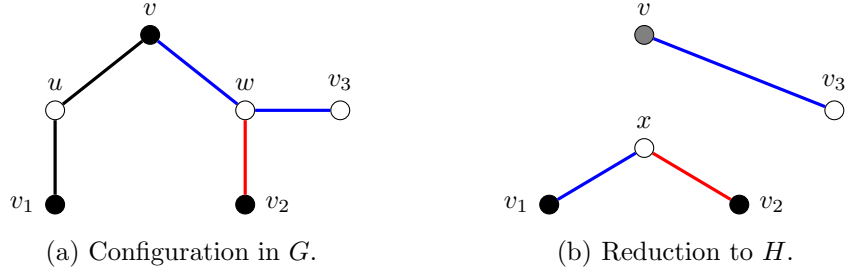
\begin{figure}[ht]
	\centering
	\begin{subfigure}{0.4\textwidth}
    \centering
	\begin{tikzpicture}
	[scale=.5,auto=center,circ/.style={circle,draw, minimum size=1pt, inner sep=2.5pt, fill=black}, >=stealth]
    	\node[circ, label={\small $v$}] (v) at (0,0) {};
    	\node[circ, label={180:\small $v_1$}] (v1) at (-2.5,-4.5) {};
    	\node[circ, label={0:\small $v_2$}] (v2) at (2.5,-4.5) {};
    	\node[circ, fill=white, label={\small $v_3$}] (v3) at (5,-2) {};
    	\node[circ, fill=white, label={\small $u$}] (u) at (-2.5,-2) {};
    	\node[circ, fill=white, label={\small $w$}] (w) at (2.5,-2) {};
    	\begin{scope}[every edge/.style={very thick,draw=red}]
    	\draw (v2) edge (w);
    	\end{scope}
    	\begin{scope}[every edge/.style={very thick,draw=blue}]
    	\draw (v) edge (w);
    	\draw (v3) edge (w);
    	\end{scope}
    	\begin{scope}[every edge/.style={very thick,draw=black}]
    	\draw (v) edge (u);
    	\draw (v1) edge (u);
    	\end{scope}
	\end{tikzpicture}
	\caption{Configuration in $G$.}
    \end{subfigure}
	\begin{subfigure}{0.4\textwidth}
    \centering
	\begin{tikzpicture}
	[scale=.5,auto=center,circ/.style={circle,draw, minimum size=1pt, inner sep=2.5pt, fill=black}, >=stealth]
    	\node[circ, fill=gray, label={\small $v$}] (v) at (0,0) {};
    	\node[circ, label={180:\small $v_1$}] (v1) at (-2.5,-4.5) {};
    	\node[circ, label={0:\small $v_2$}] (v2) at (2.5,-4.5) {};
    	\node[circ, fill=white, label={\small $v_3$}] (v3) at (5,-2) {};
    	\node[circ, fill=white, label={\small $x$}] (x) at (0,-3) {};
    	\begin{scope}[every edge/.style={very thick,draw=red}]
    	\draw (v2) edge (x);
    	\end{scope}
    	\begin{scope}[every edge/.style={very thick,draw=blue}]
    	\draw (v3) edge (v);
    	\draw (v1) edge (x);
    	\end{scope}
    	\begin{scope}[every edge/.style={very thick,draw=black}]
    	\end{scope}
	\end{tikzpicture}
	\caption{Reduction to $H$.}
	\label{fig:redcrit23}
    \end{subfigure}
	\caption{Reduction step in Lemma~\ref{lem:7verts}.(iii).}
\end{figure}
Apply the standard reduction for $2$-vertices and $3$-vertices to $u$ and $w$ respectively, and add a new $2$-vertex $x$ joined to $v_1$ and $v_2$, yielding a double tree $H$ as shown in Figure~\ref{fig:redcrit23}. By induction, $H$ has a balanced double tree decomposition $H=S_1\sqcup S_2$.
By symmetry we may assume that $vv_3\in E(S_1)$.
We treat two cases separately.

\textit{Case 1.} 
$xv_2\in E(S_2)$.

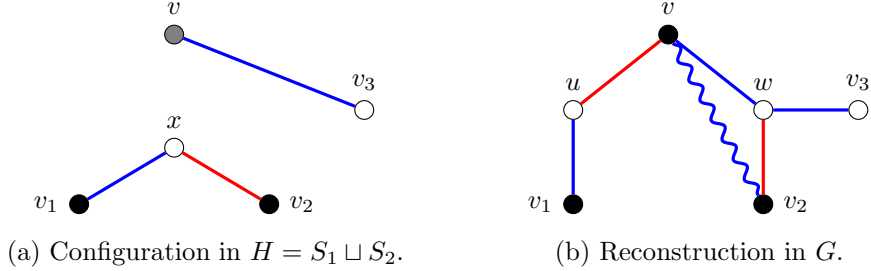
\begin{figure}[ht]
	\centering
    \begin{subfigure}{0.4\textwidth}
    \centering
	\begin{tikzpicture}
	[scale=.5,auto=center,circ/.style={circle,draw, minimum size=1pt, inner sep=2.5pt, fill=black}, >=stealth]
	\begin{scope}
    	\node[circ, fill=gray, label={\small $v$}] (v) at (0,0) {};
    	\node[circ, label={180:\small $v_1$}] (v1) at (-2.5,-4.5) {};
    	\node[circ, label={0:\small $v_2$}] (v2) at (2.5,-4.5) {};
    	\node[circ, fill=white, label={\small $v_3$}] (v3) at (5,-2) {};
    	\node[circ, fill=white, label={\small $x$}] (x) at (0,-3) {};
    	\begin{scope}[every edge/.style={very thick,draw=red}]
    	\draw (v2) edge (x);
    	\end{scope}
    	\begin{scope}[every edge/.style={very thick,draw=blue}]
    	\draw (v3) edge (v);
    	\draw (v1) edge (x);
    	\end{scope}
    	\begin{scope}[every edge/.style={very thick,draw=black}]
    	\end{scope}
	\end{scope}
	\end{tikzpicture}
	\caption{Configuration in $H = S_1\sqcup S_2$.} 
    \end{subfigure}
    \begin{subfigure}{0.4\textwidth}
    \centering
	\begin{tikzpicture}
	[scale=.5,auto=center,circ/.style={circle,draw, minimum size=1pt, inner sep=2.5pt, fill=black}, >=stealth]
	
	\begin{scope}
    	\node[circ, label={\small $v$}] (v) at (0,0) {};
    	\node[circ, label={180:\small $v_1$}] (v1) at (-2.5,-4.5) {};
    	\node[circ, label={0:\small $v_2$}] (v2) at (2.5,-4.5) {};
    	\node[circ, fill=white, label={\small $v_3$}] (v3) at (5,-2) {};
    	\node[circ, fill=white, label={\small $u$}] (u) at (-2.5,-2) {};
    	\node[circ, fill=white, label={\small $w$}] (w) at (2.5,-2) {};
    	\begin{scope}[every edge/.style={very thick,draw=red}]
    	\draw (v2) edge (w);
    	\draw (v) edge (u);
    	\end{scope}
    	\begin{scope}[every edge/.style={very thick,draw=blue}]
    	\draw (v) edge (w);
    	\draw (v3) edge (w);
    	\draw (v1) edge (u);
    	\draw (v) edge[ snake=snake,
    	        segment amplitude=.6mm,
	            segment length=3mm,
	            line after snake=0mm] (v2);
    	\end{scope}
    	\begin{scope}[every edge/.style={very thick,draw=black}]
    	\end{scope}
	\end{scope}
	\end{tikzpicture}
	\caption{Reconstruction in $G$.}
	\label{fig:reccrit31b}
    \end{subfigure}
	\caption{Reconstruction step in Case 1.}
\end{figure}

Reverse the reductions to get $G= T_1'\sqcup T_2'$, where
\begin{align*}
E(T_1') &= E(S_1) - xv_1 - vv_3 + vw + wv_3 +uv_1,\\
E(T_2') &= E(S_2) - xv_2 + wv_2 + uv,
\end{align*}
as in Figure~\ref{fig:reccrit31b}.
We claim that every vertex in $T_1'\sqcup T_2'$ is balanced.
All degree differences are unchanged except at \(v\) where a red edge has been added.
By Lemma~\ref{fact:poor3}, there is a path in $G$ from $v_2$ to $v$ in $T_1'$ that does not pass through $v_3$, so $d_{T_1'}(v) \geq 2$.
As $v$ is critical, it is balanced, as required.

\emph{Case 2.}
\(xv_2 \in E(S_1)\)
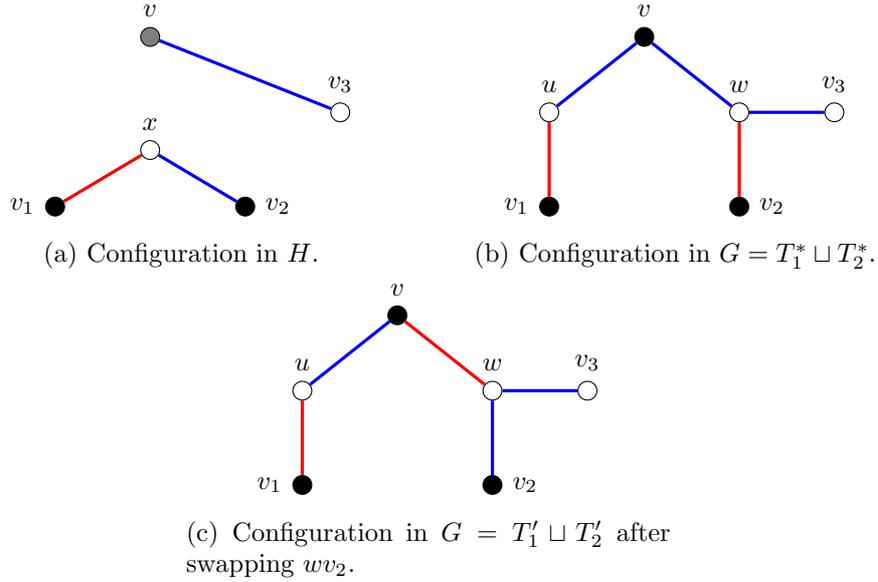
\begin{figure}[ht] 
	\centering
	\begin{subfigure}{0.4\textwidth}
    \centering
	\begin{tikzpicture}
	[scale=.5,auto=center,circ/.style={circle,draw, minimum size=1pt, inner sep=2.5pt, fill=black}, >=stealth]
    	\node[circ, fill=gray, label={\small $v$}] (v) at (0,0) {};
    	\node[circ, label={180:\small $v_1$}] (v1) at (-2.5,-4.5) {};
    	\node[circ, label={0:\small $v_2$}] (v2) at (2.5,-4.5) {};
    	\node[circ, fill=white, label={\small $v_3$}] (v3) at (5,-2) {};
    	\node[circ, fill=white, label={\small $x$}] (x) at (0,-3) {};
    	\begin{scope}[every edge/.style={very thick,draw=red}]
    	\draw (v1) edge (x);
    	\end{scope}
    	\begin{scope}[every edge/.style={very thick,draw=blue}]
    	\draw (v3) edge (v);
    	\draw (v2) edge (x);
    	\end{scope}
    	\begin{scope}[every edge/.style={very thick,draw=black}]
    	\end{scope}
	\end{tikzpicture}
	\caption{Configuration in $H$.} 
	\end{subfigure}
	\begin{subfigure}{0.4\textwidth}
    \centering
	\begin{tikzpicture}
	[scale=.5,auto=center,circ/.style={circle,draw, minimum size=1pt, inner sep=2.5pt, fill=black}, >=stealth]
    	\node[circ, label={\small $v$}] (v) at (0,0) {};
    	\node[circ, label={180:\small $v_1$}] (v1) at (-2.5,-4.5) {};
    	\node[circ, label={0:\small $v_2$}] (v2) at (2.5,-4.5) {};
    	\node[circ, fill=white, label={\small $v_3$}] (v3) at (5,-2) {};
    	\node[circ, fill=white, label={\small $u$}] (u) at (-2.5,-2) {};
    	\node[circ, fill=white, label={\small $w$}] (w) at (2.5,-2) {};
    	\begin{scope}[every edge/.style={very thick,draw=red}]
    	\draw (v2) edge (w);
    	\draw (v1) edge (u);
    	\end{scope}
    	\begin{scope}[every edge/.style={very thick,draw=blue}]
    	\draw (v) edge (u);
    	\draw (v3) edge (w);
    	\draw (v) edge (w);
    	        segment amplitude=.6mm,
	            segment length=3mm,
	            line after snake=0mm] (v2);
    	\end{scope}
    	\begin{scope}[every edge/.style={very thick,draw=black}]
    	\end{scope}
	\end{tikzpicture}
	\caption{Configuration in $G= T_1^{\ast}\sqcup T_2^{\ast}$.} 
	\label{fig:reccrit32iib}
	\end{subfigure}
	\begin{subfigure}{0.4\textwidth}
    \centering
	\begin{tikzpicture}
	[scale=.5,auto=center,circ/.style={circle,draw, minimum size=1pt, inner sep=2.5pt, fill=black}, >=stealth]
    	\node[circ, label={\small $v$}] (v) at (0,0) {};
    	\node[circ, label={180:\small $v_1$}] (v1) at (-2.5,-4.5) {};
    	\node[circ, label={0:\small $v_2$}] (v2) at (2.5,-4.5) {};
    	\node[circ, fill=white, label={\small $v_3$}] (v3) at (5,-2) {};
    	\node[circ, fill=white, label={\small $u$}] (u) at (-2.5,-2) {};
    	\node[circ, fill=white, label={\small $w$}] (w) at (2.5,-2) {};
    	\begin{scope}[every edge/.style={very thick,draw=red}]
    	\draw (v) edge (w);
    	\draw (v1) edge (u);
    	\end{scope}
    	\begin{scope}[every edge/.style={very thick,draw=blue}]
    	\draw (v) edge (u);
    	\draw (v3) edge (w);
    	\draw (v2) edge (w);
    	        segment amplitude=.6mm,
	            segment length=3mm,
	            line after snake=0mm] (v2);
    	\end{scope}
    	\begin{scope}[every edge/.style={very thick,draw=black}]
    	\end{scope}
	\end{tikzpicture}
	\caption{Configuration in $G=T_1'\sqcup T_2'$ after swapping $wv_2$.} \label{fig:reccrit32iic}
	\end{subfigure}
	\caption{Reconstruction in Case 2.}
\end{figure}

Reverse the reductions to get $G = T_1^{\ast}\sqcup T_2^{\ast}$, where
\begin{align*}
E(T_1^{\ast}) &= E(S_1) - xv_2 - vv_3 + uv + wv_3+ vw ,\\
E(T_2^{\ast}) &= E(S_2) - xv_1 + uv_1 +wv_2,
\end{align*}
as in Figure~\ref{fig:reccrit32iib}.
After swapping $wv_2$ we obtain, by Lemma~\ref{fact:poor3}, the decomposition $G = T_1'\sqcup T_2'$ shown in Figure~\ref{fig:reccrit32iic}.
Since $u$ and $w$ are both leaves in $T_1'$, $T_2'$ respectively with $vu\in E(T_1')$ and $vw\in E(T_2')$, we have $d_{T_i'}(v) \geq 2$ for $i=1,2$ and therefore $v$ is balanced as it is critical.
All other degree differences at big vertices remain unchanged. Hence, $G=T_1'\sqcup T_2'$ is balanced, a contradiction.
\end{proof}

\subsection{Discharging}\label{subsec:discharge}
In this section we conclude the proof of Theorem~\ref{thm:2trees} by applying the lemmas above with $c = 4$.
\begin{proof}[Proof of Theorem~\ref{thm:2trees}] 
Let $G$ be a counterexample minimising the number of vertices $n$.
Define the initial charge function $f\colon V \to \Q$ by $f(v) = d(v)$ and the \emph{discharging procedure} as follows.
For each big vertex $v$ and each edge $vu$ it is incident to, send to the vertex $u$
\begin{itemize}[noitemsep]
\item charge $1$ if $u$ is a $2$-vertex,
\item charge $1/2$ if $u$ is a poor $3$-vertex, 
\item charge $1/2$ if $u$ is a bad $3$-vertex (note that a bad $3$-vertex receives a total charge of $1$ from $v$ because of the double edge), 
\item charge $1/3$ if $u$ is a rich $3$-vertex.
\end{itemize}
Let $g\colon V\to \Q$ be the charge function after the discharging procedure has taken place.
Then 
\begin{equation*}
\sum_{v\in V(G)} g(v) = \sum_{v\in V(G)} f(v) = 4 n -  4.
\end{equation*}
We claim that every vertex $v$ of $G$ has $g(v) \geq 4$, which will give a contradiction.

Indeed, 
if $d(v)\geq 9$, then, by Lemma~\ref{lem:atmostone2vert}, $v$ is adjacent to at most one $2$-vertex and therefore
\begin{equation*}
g(v) \geq 9 - (1 + 8\cdot \tfrac{1}{2}) = 4.
\end{equation*}
If $d(v) = 8$, then, by Lemma~\ref{lem:2verts}, $v$ cannot be adjacent to any $2$-vertices and therefore
\begin{equation*}
g(v) \geq 8 - 8\cdot\tfrac{1}{2}  = 4.
\end{equation*}
If $d(v) = 7$, we distinguish two cases.
\begin{itemize}[noitemsep]
\item If $\Gamma(v)$ contains only small vertices, then by Lemma~\ref{lem:7verts}.(i), $v$ does not have bad vertices in its neighbourhood, and by Lemmas~\ref{lem:2verts}, \ref{lem:7verts}.(ii), \ref{lem:7verts}.(iii), $v$ has at most one neighbour that is a $2$-vertex or a poor $3$-vertex.
Therefore,
$g(v) \geq 7- 1 - 6\cdot 1/3 = 4$.
\item If $\Gamma(v)$ has a big vertex, using Lemmas~\ref{lem:2verts}, \ref{lem:badvertices}, \ref{lem:7verts}.(ii), \ref{lem:7verts}.(iii), we similarly get that
$g(v) \geq 7- 1 - 2\cdot 1/2 - 3\cdot 1/3 =4$.
\end{itemize}

If $d(v) \in \{4,5,6\}$, then $g(v) = d(v) \geq 4$.

If $d(v) =3$, there are two cases.
\begin{enumerate}[noitemsep]
\item If $v$ is a rich $3$-vertex, it receives a charge of $1/3$ from each of its edges, thus $g(v) = 3+ 3\cdot 1/3 = 4$.
\item If $v$ is a poor or bad $3$-vertex, it receives a charge of $1/2$ from two of its edges, thus $g(v) = 3 + 2 \cdot 1/2 = 4$.
\end{enumerate}
If $d(v) = 2$, then by Lemma~\ref{lem:2verts}, $v$ receives a charge of $1$ from both its neighbours and thus $g(v) = 4$, as required.
\end{proof}

\section{General graphs}\label{sec:general}
In this section we write $G= A+M$ to mean that $G = A\sqcup M$ where $A$ is a spanning double tree and $M$ a graph.

We deduce Theorem~\ref{thm:2treesgeneral} from a slightly more general statement.
\begin{theorem}\label{claim:general}
Let $G = A + M$. 
Then $G$ admits a $4$-balanced decomposition into subgraphs $G=G_1\sqcup G_2$ such that
$(A\cap G_1) \sqcup (A \cap G_2)$ is a double tree decomposition of $A$.
\end{theorem} 
Theorem~\ref{claim:general} follows from similar arguments to those used for Theorem~\ref{thm:2trees}, with suitable modifications.

Fix an integer $c\geq 2$ and suppose there are graphs $G = A + M$ with no \(c\)-balanced decomposition $G=G_1\sqcup G_2$ such that
$A\cap G_1$ and $A \cap G_2$ are trees. We take \(G\) to be such a graph where
\begin{enumerate}[noitemsep]
\item $e(M)$ is minimal,
\item subject to this, $|G|$ is minimal.
\end{enumerate}
Again, define $v\in V(G)$ to be \emph{big} if $d_G(v) \geq c+3$ and \emph{small} otherwise. 
If $d_G(v) = c+3$ we again call $v$ \emph{critical}.

\textit{In figures, edges of $M$ are dashed.}
\subsection{Edges of \texorpdfstring{$M$}{M}}
\begin{lemma}\label{lem:Mparity} 
If $e\in E(M)$ is incident to a vertex $v$, then $v$ is big 
and $d_G(v) \equiv c+1\bmod{2}$.
\end{lemma}
\begin{proof}
Let $uv\in E(M)$.
Remove $uv$, rebalance the resulting graph using minimality of $G$ and add $uv$ to the appropriate part so that $u$ is balanced in the resulting decomposition $G=G_1\sqcup G_2$.
By construction we further have that $(A\cap G_1 ) \sqcup (A\cap G_2)$ is a double tree decomposition.
All degree differences have been preserved at vertices of $G$ other than $u$ or $v$, and $u$ is balanced in $G_1\sqcup G_2$ by construction. Hence, the vertex \(v\) cannot be balanced in \(G_1 \sqcup G_2\).

If $v$ is small, then $v$ is clearly balanced in $G=G_1\sqcup G_2$.
If $d_{G}(v) \equiv c \bmod{2}$, then a parity argument similar to that of Lemma~\ref{lem:2verts} shows that $v$ is balanced.
Thus, neither can occur.
\end{proof}
As a consequence, the edges of $M$ are not incident to any $3$-vertices.
We will use the terminology of rich, poor and bad \(3\)-vertices defined in Section~\ref{sec:3verts}. When we do edge swaps we will do them within the double tree \(A\).

Note that all edges appearing in Lemmas~\ref{lem:2verts}, \ref{lem:atmostone2vert} and \ref{lem:3vertsbig}-\ref{lem:7verts} are incident to a small vertex and so are in the double tree $A$ by Lemma~\ref{lem:Mparity}.
Hence these lemmas all still hold in \(G\).
Indeed, reductions and reconstructions are unchanged when the vertices involved edges are in $M$.
For our purposes we require a slight strengthening of Lemma~\ref{lem:7verts}.(i) (this follows immediately from the proof of Lemma~\ref{lem:7verts} when applied in this context).
\begin{lemma}\label{lem:Mbad} 
Let $v\in V(G)$ be a critical vertex.
If all neighbours of $v$ in $A$ are small, then $v$ is not adjacent to any bad $3$-vertex.
\end{lemma}
Edges of $M$ are subject to further constraints, which we will need in the discharging argument.
\begin{lemma}\label{lem:Mmatching}
The subgraph $M$ is a matching.
\end{lemma}
\begin{proof}
Let $u,v,w\in V(G)$ and suppose that $uv, vw\in E(M)$.
Then define $H = B  + N$ 
by deleting edges $uv$, $vw$ from $M$ (to give $N$), and by adding a new $2$-vertex $x$ joined to both $u$ and $w$ in the double tree (to give $B$).
Then $e(N)<e(M)$ and so by minimality we may find a balanced decomposition $H = H_1 \sqcup H_2$, where $(B\cap H_1) \sqcup (B\cap H_2)$ is a double tree.
By symmetry we may assume $ux \in E( H_1)$, $xw \in E(H_2)$.
Define $G=  G_1\sqcup G_2$ where $E(G_1) = E(H_1) -ux + uv$ and $E(G_2) = E(H_2) - xw + vw$.
Then $G= G_1\sqcup G_2$ is a balanced decomposition as degree differences have been preserved.
Also, $(A\cap G_1)\sqcup (A\cap G_2)$ is a double tree as $(B\cap H_1) \sqcup (B\cap H_2)$ was, giving a contradiction.
\end{proof}
\begin{lemma}\label{lem:M2verts}
Let $v\in V(G)$.
Then $v$ cannot be both adjacent to a $2$-vertex and incident to an edge of $M$.
\end{lemma}
\begin{proof}

Suppose that $u,v\in V(G)$, $uv\in E(M)$ and $v$ is adjacent to a $2$-vertex $w$.
Let $v'\neq v$ be the other neighbour of $w$. By Lemma~\ref{lem:Mparity}, both \(u\) and \(v\) are big.

\textit{Case 1.} $u = v'$.

\begin{figure}[ht] 
	\centering
    \begin{subfigure}[c][][c]{0.4\textwidth}
    \centering
	\begin{tikzpicture}
	[scale=.5,auto=center,circ/.style={circle,draw, minimum size=1pt, inner sep=2.5pt, fill=black}, >=stealth]
    	\node[circ, fill=black, label={\small $v$}] (v) at (0,0) {};
    	\node[circ, fill=white, label={\small $w$}] (w) at (2.5,-2) {};
    	\node[circ, label={\small $u$}] (u) at (-2.5,-2) {};
    	\begin{scope}[every edge/.style={very thick,draw=red}]
    	\end{scope}
    	\begin{scope}[every edge/.style={very thick,draw=blue}]
    	\end{scope}
    	\begin{scope}[every edge/.style={very thick,draw=black}]
    	\draw (u) edge[dashed] (v);
    	\draw (v) edge (w);
    	\draw (u) edge (w);
    	\end{scope}
	\end{tikzpicture}
	\caption{Configuration in $G$.} 
	\label{fig:redgeneral2a}
    \end{subfigure}
    \begin{subfigure}{0.4\textwidth}
    \centering
	\begin{tikzpicture}
	[scale=.5,auto=center,circ/.style={circle,draw, minimum size=1pt, inner sep=2.5pt, fill=black}, >=stealth]
    	\node[circ, fill=black, label={\small $v$}] (v) at (0,0) {};
    	\node[fill=white, draw=white] (w) at (2.5,-2) {};
    	\node[circ, label={\small $u$}] (u) at (-2.5,-2) {};
    	\begin{scope}[every edge/.style={very thick,draw=red}]
    	\end{scope}
    	\begin{scope}[every edge/.style={very thick,draw=blue}]
    	\end{scope}
    	\begin{scope}[every edge/.style={very thick,draw=black}]
    	\end{scope}
	\end{tikzpicture}
	\caption{Reduction to $H$.}
    \end{subfigure}
    \caption{Reduction step in Lemma~\ref{lem:M2verts} when $u= v'$.}
\end{figure}
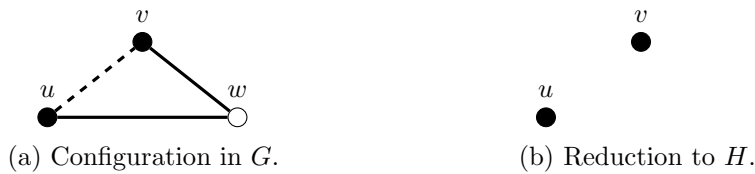
Define $H = B + N$ by deleting $w$ (so \(A\) becomes \(B\)) and removing $uv$ (so \(M\) becomes $N$).
Then $e(N) < e(M)$ and so, by minimality, there is a balanced decomposition $H= H_1 \sqcup H_2$ where $(B\cap H_1) \sqcup (B\cap H_2) $ is a double tree decomposition.
Without loss of generality we may assume that $d_{H_1}(v)\leq d_{H_2}(v)$.

Define $G = G_1\sqcup G_2$ where $E(G_1)= E(H_1) + uv + vw$ and $E(G_2) = E(H_2) + wu$.
Then as $d_{H_1}(v) \leq d_{H_2}(v)$, the vertex $v$ is balanced in $G_1\sqcup G_2$.
Since all other degree differences have been preserved, the decomposition $G = G_1\sqcup G_2$ is balanced.
Further, $(A\cap G_1)\sqcup (A\cap G_2)$ is a double tree, giving a contradiction.

\textit{Case 2.} $u \neq v'$.

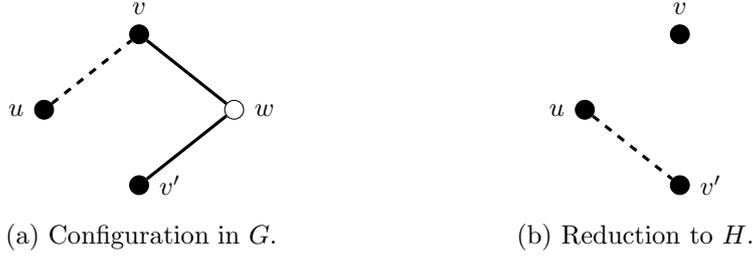
\begin{figure}[ht] 
	\centering
    \begin{subfigure}[c][][c]{0.4\textwidth}
    \centering
	\begin{tikzpicture}
	[scale=.5,auto=center,circ/.style={circle,draw, minimum size=1pt, inner sep=2.5pt, fill=black}, >=stealth]
    	\node[circ, fill=black, label={\small $v$}] (v) at (0,0) {};
    	\node[circ, label={0:\small $v'$}] (v2) at (0,-4) {};
    	\node[circ, fill=white, label={0:\small $w$}] (w) at (2.5,-2) {};
    	\node[circ, label={180:\small $u$}] (u) at (-2.5,-2) {};
    	\begin{scope}[every edge/.style={very thick,draw=black}]
    	\draw (u) edge[dashed] (v);
    	\draw (v) edge (w);
    	\draw (v2) edge (w);
    	\end{scope}
	\end{tikzpicture}
	\caption{Configuration in $G$.} 
	\label{fig:redgeneral2b}
    \end{subfigure} 
    \begin{subfigure}{0.4\textwidth}
    \centering
	\begin{tikzpicture}
	[scale=.5,auto=center,circ/.style={circle,draw, minimum size=1pt, inner sep=2.5pt, fill=black}, >=stealth]
    	\node[circ, fill=black, label={\small $v$}] (v) at (0,0) {};
    	\node[circ, label={0:\small $v'$}] (v2) at (0,-4) {};
    	\node[circ, label={180:\small $u$}] (u) at (-2.5,-2) {};
    	\begin{scope}[every edge/.style={very thick,draw=black}]
    	\draw (u) edge[dashed] (v2);
    	\end{scope}
	\end{tikzpicture}
	\caption{Reduction to $H$.}
	\label{fig:match2vertsneq}
    \end{subfigure}
    \caption{Reduction step in Lemma~\ref{lem:M2verts} when $u\neq v'$.}
\end{figure}
Define $H = B + N$ by deleting $w$ (so \(A\) becomes $B$), removing $uv$ from \(M\) and adding $uv'$ (so \(M\) becomes \(N\)), as in Figure~\ref{fig:match2vertsneq}.
Then $e(N) = e(M)$ and $|H| < |G|$ so by minimality there is a balanced decomposition $H= H_1 \sqcup H_2$ where $(B\cap H_1) \sqcup (B\cap H_2) $ is a double tree decomposition.
Without loss of generality, $uv'\in E(H_1)$.

Define $G = G_1 \sqcup G_2$ where $E(G_1) = E(H_1) - uv' + uv + wv'$, $E(G_2) = E(H_2) + vw$. 
Note that since $B\cap H_1$ and $B\cap H_2$ are both connected, the two subgraphs $A\cap G_1$ and $A \cap G_2$ are connected.
The decomposition is balanced as $H = H_1\sqcup H_2$ is balanced and degree differences are preserved, 
a contradiction.
\end{proof}
\begin{lemma}\label{lem:Mcrit}
Let $v\in V(G)$ be a critical vertex. 
Then $v$ cannot be both adjacent to a poor $3$-vertex and incident to an edge of $M$.
\end{lemma}
\begin{proof}
\begin{figure}[ht] 
	\centering
    \begin{subfigure}{0.4\textwidth}
    \centering
	\begin{tikzpicture}
	[scale=.5,auto=center,circ/.style={circle,draw, minimum size=1pt, inner sep=2.5pt, fill=black}, >=stealth]
    	\node[circ, fill=black, label={\small $v$}] (v) at (0,0) {};
    	\node[circ, fill=white, label={0:\small $s$}] (v3) at (3,-3) {};
    	\node[circ, fill=white, label={45:\small $w$}] (w) at (0,-3) {};
    	\node[circ, label={180:\small $u$}] (u) at (-3,-3) {};
    	\begin{scope}[every edge/.style={very thick,draw=black}]
    	\draw (u) edge[dashed] (v);
    	\draw (v3) edge (w);
    	\draw (v) edge (w);
    	\draw (u) edge (w);
    	\end{scope}
	\end{tikzpicture}
	\caption{Configuration in $G$.}
    \end{subfigure}
    \begin{subfigure}{0.4\textwidth}
    \centering
	\begin{tikzpicture}
	[scale=.5,auto=center,circ/.style={circle,draw, minimum size=1pt, inner sep=2.5pt, fill=black}, >=stealth]
    	\node[circ, fill=black, label={\small $v$}] (v) at (0,0) {};
    	\node[circ, fill=white, label={0:\small $s$}] (v3) at (3,-3) {};
    	\node[circ, label={180:\small $u$}] (u) at (-3,-3) {};
    	\begin{scope}[every edge/.style={very thick,draw=black}]
    	\draw (v3) edge (v);
    	\end{scope}
	\end{tikzpicture}
	\caption{Reduction to $H$.}
	\label{fig:redgeneral3eqb}
    \end{subfigure}
    \caption{Reduction step in Lemma~\ref{lem:Mcrit} when $u = v'$.}
\end{figure}
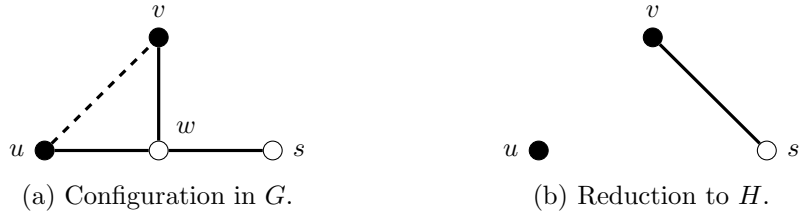
Suppose that $u,v\in V(G)$ where $v$ is critical, $uv\in E(M)$ and $v$ is adjacent to a poor $3$-vertex $w$. By Lemma~\ref{lem:Mparity}, both \(u\) and \(v\) are big.
Let $s$ be the small neighbour of $w$ and $v'\neq v$ be the other big neighbour of $w$.

First suppose that $u = v'$. By edge flipping and Lemma~\ref{fact:poor3} we may assume that $vw$ and $ws$ are in the same tree.  We carry out the standard reduction for $3$-vertices at $w$ so that $A$ becomes a double tree $B$ and delete $uv$ from $M$ to get $N$-- see Figure~\ref{fig:redgeneral3eqb}. Let $H = B + N$.

Now $e(N) < e(M)$, so by minimality there is a balanced decomposition $H= H_1 \sqcup H_2$ where $(B\cap H_1) \sqcup (B\cap H_2) $ is a double tree decomposition.
Without loss of generality, $vs\in E(B)\cap E(H_1)$.
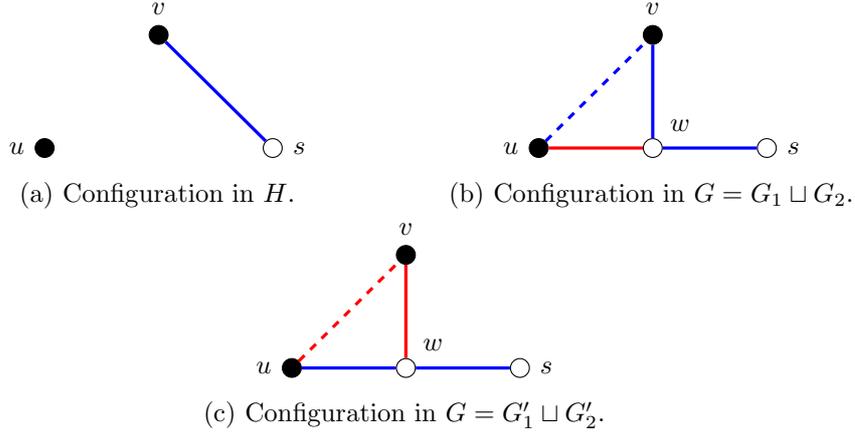
\begin{figure}[ht] 
	\centering
    \begin{subfigure}{0.4\textwidth}
    \centering
	\begin{tikzpicture}
	[scale=.5,auto=center,circ/.style={circle,draw, minimum size=1pt, inner sep=2.5pt, fill=black}, >=stealth]
    	\node[circ, fill=black, label={\small $v$}] (v) at (0,0) {};
    	\node[circ, fill=white, label={0:\small $s$}] (v3) at (3,-3) {};
    	\node[circ, label={180:\small $u$}] (u) at (-3,-3) {};
    	\begin{scope}[every edge/.style={very thick,draw=blue}]
    	\draw (v3) edge (v);
    	\end{scope}
	\end{tikzpicture}
	\caption{Configuration in $H$.}
    \end{subfigure}
    \begin{subfigure}{0.4\textwidth}
    \centering
	\begin{tikzpicture}
	[scale=.5,auto=center,circ/.style={circle,draw, minimum size=1pt, inner sep=2.5pt, fill=black}, >=stealth]
    	\node[circ, fill=black, label={\small $v$}] (v) at (0,0) {};
    	\node[circ, fill=white, label={0:\small $s$}] (v3) at (3,-3) {};
    	\node[circ, fill=white, label={45:\small $w$}] (w) at (0,-3) {};
    	\node[circ, label={180:\small $u$}] (u) at (-3,-3) {};
    	\begin{scope}[every edge/.style={very thick,draw=red}]
    	\draw (u) edge (w);
    	\end{scope}
    	\begin{scope}[every edge/.style={very thick,draw=blue}]
    	\draw (u) edge[dashed] (v);
    	\draw (v3) edge (w);
    	\draw (v) edge (w);
    	\end{scope}
	\end{tikzpicture}
	\caption{Configuration in $G=G_1\sqcup G_2$.}
	\label{fig:recgeneraleq3b}
    \end{subfigure} 
    \begin{subfigure}{0.4\textwidth}
    \centering
	\begin{tikzpicture}
	[scale=.5,auto=center,circ/.style={circle,draw, minimum size=1pt, inner sep=2.5pt, fill=black}, >=stealth]
    	\node[circ, fill=black, label={\small $v$}] (v) at (0,0) {};
    	\node[circ, fill=white, label={0:\small $s$}] (v3) at (3,-3) {};
    	\node[circ, fill=white, label={45:\small $w$}] (w) at (0,-3) {};
    	\node[circ, label={180:\small $u$}] (u) at (-3,-3) {};
    	\begin{scope}[every edge/.style={very thick,draw=red}]
    	\draw (v) edge (w);
    	\draw (u) edge[dashed] (v);
    	\end{scope}
    	\begin{scope}[every edge/.style={very thick,draw=blue}]
    	\draw (v3) edge (w);
    	\draw (u) edge (w);
    	\end{scope}
	\end{tikzpicture}
	\caption{Configuration in $G=G_1'\sqcup G_2'$.}
	\label{fig:recgeneral3eqc}
    \end{subfigure}
    \caption{Reconstruction step in Lemma~\ref{lem:Mcrit} when $u = v'$.}
\end{figure}

Define $G = G_1 \sqcup G_2$ by 
\begin{align*}
E(G_1) &= E(H_1) - vs + uv + vw + ws, \\
E(G_2) &= E(H_2) + uw,
\end{align*}
as in Figure~\ref{fig:recgeneraleq3b}.
Degree differences at all vertices of $G$ except $v$ have been preserved and $(A\cap G_1)\sqcup (A\cap G_2)$ is a double tree.
Hence, if $v$ is balanced we have a contradiction.
If $d_{G_2}(v)\geq 2$ then $v$ is balanced.
Otherwise, swap the edge $uw$ in the double tree decomposition $A= (A\cap G_1)\sqcup (A\cap G_2)$ to obtain $A = T_1 \sqcup T_2$.
Define $G=G_1'\sqcup G_2'$ where 
\begin{align*}
E(G_1') &= E(T_1) \sqcup (E(M)\cap G_1) - uv \\
E(G_2') &= E(T_2) \sqcup (E(M)\cap G_2) + uv,
\end{align*}
as in Figure~\ref{fig:recgeneral3eqc}.
Then degree differences at all vertices of $G$ except $v$ have been preserved and $(A\cap G_1')\sqcup (A\cap G_2')$ is a double tree, but now $d_{G_2}(v) \geq 2$, giving a contradiction.

We may therefore assume $u\neq v'$.
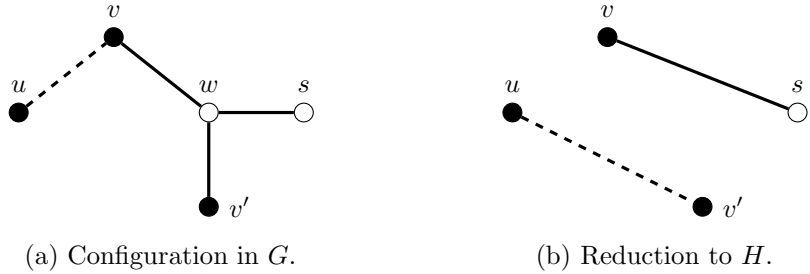
\begin{figure}[ht] 
	\centering
    \begin{subfigure}{0.4\textwidth}
    \centering
	\begin{tikzpicture}
	[scale=.5,auto=center,circ/.style={circle,draw, minimum size=1pt, inner sep=2.5pt, fill=black}, >=stealth]
    	\node[circ, fill=black, label={\small $v$}] (v) at (0,0) {};
    	\node[circ, label={0:\small $v'$}] (v2) at (2.5,-4.5) {};
    	\node[circ, fill=white, label={\small $s$}] (v3) at (5,-2) {};
    	\node[circ, fill=white, label={\small $w$}] (w) at (2.5,-2) {};
    	\node[circ, label={\small $u$}] (u) at (-2.5,-2) {};
    	\begin{scope}[every edge/.style={very thick,draw=black}]
    	\draw (u) edge[dashed] (v);
    	\draw (v3) edge (w);
    	\draw (v) edge (w);
    	\draw (v2) edge (w);
    	\end{scope}
	\end{tikzpicture}
	\caption{Configuration in $G$.} 
	\label{fig:redgeneral3a}
    \end{subfigure} 
    \begin{subfigure}{0.4\textwidth}
    \centering
	\begin{tikzpicture}
	[scale=.5,auto=center,circ/.style={circle,draw, minimum size=1pt, inner sep=2.5pt, fill=black}, >=stealth]
    	\node[circ, fill=black, label={\small $v$}] (v) at (0,0) {};
    	\node[circ, label={0:\small $v'$}] (v2) at (2.5,-4.5) {};
    	\node[circ, fill=white, label={\small $s$}] (v3) at (5,-2) {};
    	\node[circ, label={\small $u$}] (u) at (-2.5,-2) {};
    	\begin{scope}[every edge/.style={very thick,draw=black}]
    	\draw (u) edge[dashed] (v2);
    	\draw (v3) edge (v);
    	\end{scope}
	\end{tikzpicture}
	\caption{Reduction to $H$.}
	\label{fig:redgeneral3b}
    \end{subfigure}
    \caption{Reduction step in Lemma~\ref{lem:Mcrit} when $u\neq v'$.}
\end{figure}
By edge flipping and Lemma~\ref{fact:poor3} we may assume that edges $vw$ and $ws$ are in the same tree. We carry out the standard reduction for $3$-vertices at $w$ so that $A$ becomes a double tree $B$. We let $N = M - uv + uv'$. See Figure~\ref{fig:redgeneral3b}. Let $H = B + N$.

Then $e(N) = e(M)$ and $|H| < |G|$ so, by minimality, there is a balanced decomposition $H= H_1 \sqcup H_2$ where $(B\cap H_1) \sqcup (B\cap H_2) $ is a double tree decomposition.
Without loss of generality, $uv'\in E(N)\cap E(H_1)$.

\textit{Case 1.} $vs\in E(B)\cap E(H_2)$.
\begin{figure}[ht] 
	\centering
    \begin{subfigure}{0.4\textwidth}
    \centering
	\begin{tikzpicture}
	[scale=.5,auto=center,circ/.style={circle,draw, minimum size=1pt, inner sep=2.5pt, fill=black}, >=stealth]
    	\node[circ, fill=black, label={\small $v$}] (v) at (0,0) {};
    	\node[circ, label={0:\small $v'$}] (v2) at (2.5,-4.5) {};
    	\node[circ, fill=white, label={\small $s$}] (v3) at (5,-2) {};
    	\node[circ, label={\small $u$}] (u) at (-2.5,-2) {};
    	\begin{scope}[every edge/.style={very thick,draw=red}]
    	\draw (v3) edge (v);
    	\end{scope}
    	\begin{scope}[every edge/.style={very thick,draw=blue}]
    	\draw (u) edge[dashed] (v2);
    	\end{scope}
	\end{tikzpicture}
	\caption{Configuration in $H$.} 
    \end{subfigure} 
    \begin{subfigure}{0.4\textwidth}
    \centering
	\begin{tikzpicture}
	[scale=.5,auto=center,circ/.style={circle,draw, minimum size=1pt, inner sep=2.5pt, fill=black}, >=stealth]
    	\node[circ, fill=black, label={\small $v$}] (v) at (0,0) {};
    	\node[circ, label={0:\small $v'$}] (v2) at (2.5,-4.5) {};
    	\node[circ, fill=white, label={\small $s$}] (v3) at (5,-2) {};
    	\node[circ, fill=white, label={\small $w$}] (w) at (2.5,-2) {};
    	\node[circ, label={\small $u$}] (u) at (-2.5,-2) {};
    	\begin{scope}[every edge/.style={very thick,draw=red}]
    	\draw (v3) edge (w);
    	\draw (v) edge (w);
    	\draw (v) edge[ snake=snake,
    	        segment amplitude=.6mm,
	            segment length=3mm,
	            line after snake=0mm] (v2);
    	\end{scope}
    	\begin{scope}[every edge/.style={very thick,draw=blue}]
    	\draw (v2) edge (w);
    	\draw (u) edge[dashed] (v);
    	\end{scope}
	\end{tikzpicture}
	\caption{Reconstruction in $G$.}
	\label{fig:recgeneral3b}
    \end{subfigure}
    \caption{Reconstruction step in Case 1.}
\end{figure}
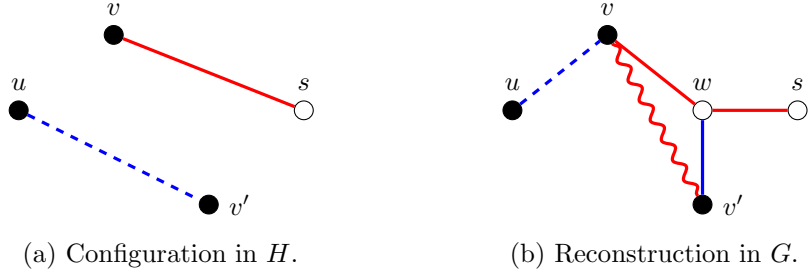

Reverse the reductions to give a decomposition $G = G_1 \sqcup G_2$ defined by
\begin{align*}
E(G_1) &= E(H_1) - uv' + uv + wv', \\
E(G_2) &= E(H_2) - vs + vw + ws,
\end{align*}
as in Figure~\ref{fig:recgeneral3b}.
Since $B\cap H_1$ and $B\cap H_2$ are both connected, $T_1 \coloneqq A\cap G_1$ and $T_2 \coloneqq A\cap G_2$ are as well and form a double tree decomposition for $A$.

All degree differences at big vertices have been preserved except at \(v\) where an extra blue edge is present.
But, by Lemma~\ref{fact:poor3}, the path from $v$ to $v'$ in $T_2$ does not contain $s$, so $d_{G_2}(v)\geq d_{T_2}(v)\geq 2$ and so $v$ is balanced in $G= G_1\sqcup G_2$ as it is critical.
Hence, $G= G_1\sqcup G_2$ is balanced, a contradiction.

\textit{Case 2.} $vs\in E(B)\cap E(H_1)$.
\begin{figure}[ht] 
	\centering
    \begin{subfigure}{0.4\textwidth}
    \centering
	\begin{tikzpicture}
	[scale=.5,auto=center,circ/.style={circle,draw, minimum size=1pt, inner sep=2.5pt, fill=black}, >=stealth]
    	\node[circ, fill=black, label={\small $v$}] (v) at (0,0) {};
    	\node[circ, label={0:\small $v'$}] (v2) at (2.5,-4.5) {};
    	\node[circ, fill=white, label={\small $s$}] (v3) at (5,-2) {};
    	\node[circ, label={\small $u$}] (u) at (-2.5,-2) {};
    	\begin{scope}[every edge/.style={very thick,draw=blue}]
    	\draw (u) edge[dashed] (v2);
    	\draw (v3) edge (v);
    	\end{scope}
	\end{tikzpicture}
	\caption{Configuration in $H$.} 
    \end{subfigure} 
    \begin{subfigure}{0.4\textwidth}
    \centering
	\begin{tikzpicture}
	[scale=.5,auto=center,circ/.style={circle,draw, minimum size=1pt, inner sep=2.5pt, fill=black}, >=stealth]
    	\node[circ, fill=black, label={\small $v$}] (v) at (0,0) {};
    	\node[circ, label={0:\small $v'$}] (v2) at (2.5,-4.5) {};
    	\node[circ, fill=white, label={\small $s$}] (v3) at (5,-2) {};
    	\node[circ, fill=white, label={\small $w$}] (w) at (2.5,-2) {};
    	\node[circ, label={\small $u$}] (u) at (-2.5,-2) {};
    	
    	\begin{scope}[every edge/.style={very thick,draw=red}]
    	\draw (v2) edge (w);
    	\end{scope}
    	\begin{scope}[every edge/.style={very thick,draw=blue}]
    	\draw (v) edge (w);
    	\draw (v3) edge (w);
    	\draw (u) edge[dashed] (v);
    	\draw (v) edge[ snake=snake,
    	        segment amplitude=.6mm,
	            segment length=3mm,
	            line after snake=0mm] (v2);
    	\end{scope}
	\end{tikzpicture}
	\caption{Configuration in $G=G_1^{\ast}\sqcup G_2^{\ast}$.} 
	\label{fig:recgeneral3ib}
    \end{subfigure}
    \begin{subfigure}{0.4\textwidth}
    \centering
	\begin{tikzpicture}
	[scale=.5,auto=center,circ/.style={circle,draw, minimum size=1pt, inner sep=2.5pt, fill=black}, >=stealth]
    	\node[circ, fill=black, label={\small $v$}] (v) at (0,0) {};
    	\node[circ, label={0:\small $v'$}] (v2) at (2.5,-4.5) {};
    	\node[circ, fill=white, label={\small $s$}] (v3) at (5,-2) {};
    	\node[circ, fill=white, label={\small $w$}] (w) at (2.5,-2) {};
    	\node[circ, label={\small $u$}] (u) at (-2.5,-2) {};
    	\begin{scope}[every edge/.style={very thick,draw=red}]
    	\draw (v) edge (w);
    	\end{scope}
    	\begin{scope}[every edge/.style={very thick,draw=blue}]
    	\draw (v2) edge (w);
    	\draw (v3) edge (w);
    	\draw (u) edge[dashed] (v);
    	\draw (v) edge[ snake=snake,
    	        segment amplitude=.6mm,
	            segment length=3mm,
	            line after snake=0mm] (v2);
    	\end{scope}
	\end{tikzpicture}
	\caption{Configuration in $G=G_1\sqcup G_2$.}
    \end{subfigure}
    \caption{Reconstruction step in Case 2.}
\end{figure}
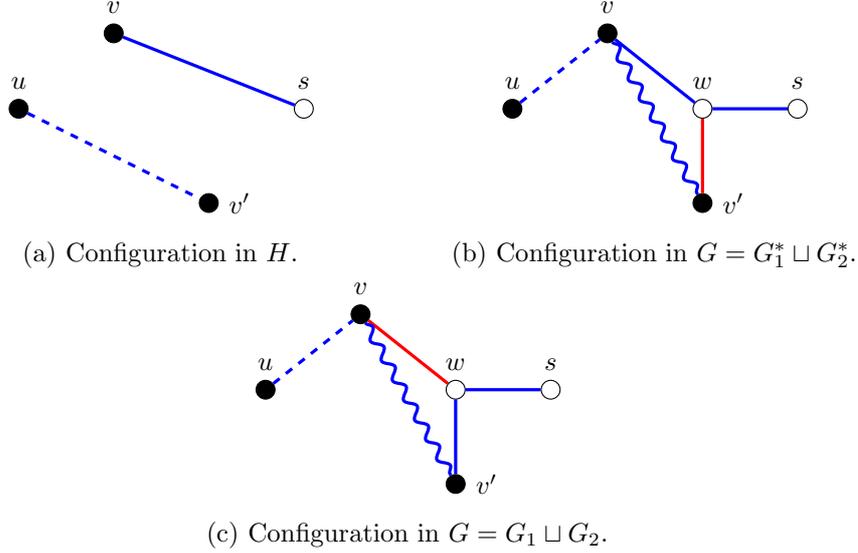

Reverse the reductions to give a decomposition $G=G_1^{\ast}\sqcup G_2^{\ast}$ where 
\begin{align*}
E(G_1^{\ast}) &= E(H_1) - uv' - vs + uv + vw + ws, \\
E(G_2^{\ast}) &= E(H_2) + wv',
\end{align*}
as in Figure~\ref{fig:recgeneral3ib}.
Since $B\cap H_1$ and $B\cap H_2$ are both connected, $T_1 \coloneqq A\cap G_1^{\ast}$ and $T_2 \coloneqq A\cap G_2^{\ast}$ are as well and form a double tree decomposition for $A$.
Let $A=S_1\sqcup S_2$ be the double tree decomposition obtained by swapping edge $wv'$ in $A$ (this swaps with $wv$ by Lemma~\ref{fact:poor3}).
Let $G = G_1\sqcup G_2$ be the decomposition where
\begin{align*}
E(G_1) &= E(N\cap G_1^{\ast}) \sqcup E(S_1),\\
E(G_2) &= E(N\cap G_2^{\ast}) \sqcup E(S_2).
\end{align*}
Then $S_1 = A\cap G_1$ and $S_2 = A\cap G_2$ are both spanning trees.
All degree differences at big vertices have been preserved except at \(v\) where an extra red edge is present. But, by Lemma~\ref{fact:poor3}, there is a path from $v$ to $v'$ in $S_1$ that does not contain $s$, so $d_{G_1}(v) \geq d_{S_1}(v) \geq 2$ and so $v$ is balanced in $G= G_1\sqcup G_2$ as it is critical.
Hence, $G = G_1\sqcup G_2$ is balanced, a contradiction.
\end{proof}

\subsection{Discharging}
\begin{proof}[Proof of Theorem~\ref{claim:general}]
Let $G = A + M$ be a counterexample to the bound $c = 4$ such that
\begin{enumerate}[noitemsep]
\item $e(M)$ is minimal,
\item subject to this, $|G|$ is minimal.
\end{enumerate}

Define the charge function $f\colon V \to \Q$ to be the degree of $v$ in the double tree $A$: $f(v) = d_A(v)$. 
Define the \emph{discharging procedure} similarly to the proof of Theorem~\ref{thm:2trees}. 
For each edge $uv \in E(A)$, a big vertex $v$ sends to its \mbox{neighbour $u$}
\begin{itemize}[noitemsep]
\item charge $1$ if $u$ is a $2$-vertex,
\item charge $1/2$ if $u$ is a poor $3$-vertex, 
\item charge $1/2$ if $u$ is a bad $3$-vertex, 
\item charge $1/3$ if $u$ is a rich $3$-vertex.
\end{itemize}
Let $g\colon V\to \Q$ be the charge function after the discharging procedure has taken place.
Then 
\begin{equation*}
\sum_{v\in V(G)} g(v) = \sum_{v\in V(G)} f(v) = 2e(A) = 4 n -  4.
\end{equation*}
We claim that every vertex $v$ of $G$ has $g(v) \geq 4$, which will give a contradiction.
As in the proof of Theorem~\ref{thm:2trees}, the claim holds if $v\in V(G)$ is not incident to any edge of $M$.
If $v$ is incident to at least one edge of $M$, then, by Lemmas~\ref{lem:Mparity}, \ref{lem:Mmatching}, \ref{lem:M2verts} and \ref{lem:Mcrit}, $v$ is big, has odd degree, is incident to exactly one $e\in E(M)$, is not adjacent to any $2$-vertices and is not adjacent to any poor $3$-vertex. There are two cases remaining:
\begin{enumerate}
\item $d(v) = k\geq 9$.
Then $g(v) \geq (k - 1) - (k-1)/2\geq 4$.
\item $d(v) = 7$.
If all neighbours of $v$ in $A$ are small, then, by Lemma~\ref{lem:Mbad}, $g(v) = 6 - 6\cdot 1/3\geq 4$.
Otherwise $v$ has a big neighbour, so $g(v) \geq 6 - 2\cdot 1/2 -  3\cdot 1/3\geq 4$.\qedhere
\end{enumerate}
\end{proof}

\section{Balancing infinite graphs}\label{sec:semiconn}
Let $G = (V,E)$ be an undirected graph.
Recall that a spanning subgraph $H\subset G$ is called \emph{semiconnected} if it contains an edge of every finite cut of $G$.
We note that this notion depends on the ambient graph $G$ and that for finite graphs, the notions of spanning connected and semiconnected subgraphs coincide.

The main results of this section are the following, which imply Theorem~\ref{cor:infinite}.
\begin{theorem}\label{thm:semiconn}
Let $c$ be minimal such that any finite double tree has a $c$-balanced decomposition. 
Then if $G$ is a countable infinite double tree, it admits a $c$-balanced decomposition $G= S_1 \sqcup S_2$ where $S_1, S_2$ are semiconnected and acyclic.
\end{theorem}
\begin{theorem}\label{thm:semiconngeneral}
Let $c$ be minimal such that any finite graph containing a spanning double tree has a $c$-balanced decomposition into connected graphs. 
Then if $G$ is a countable infinite graph containing a spanning double tree, it admits a $c$-balanced decomposition $G= S_1 \sqcup S_2$ where $S_1, S_2$ are semiconnected.
\end{theorem}
As the proof of both of these theorems is virtually the same, we only spell out a proof of the first.
\begin{proof}[Proof of Theorem~\ref{thm:semiconn}]
Without loss of generality we may assume that $G$ is locally finite.
Indeed, if $v\in V(G)$ is a vertex of infinite degree with neighbours $x_1, x_2,\dotsc $, we may replace it with a path of double edges $v_0v_1v_2 
\dotsc$ where $v_i$ is connected to every vertex in $\{x_{(c+1)i+k} \colon k \in [c+1]\}$.
Applying this reduction to every vertex of infinite degree we obtain a locally finite countable graph $H$.
If $H$ has a $c$-balanced decomposition $T_1\sqcup T_2$, every vertex $v_i$ must have at least one edge of both $T_1$ and $T_2$ that is not an edge of the path.
Hence we may reconstruct a balanced decomposition for $G$ by merging the double paths we created,
as degrees with infinite degree have infinite degree in both trees after merging. 

Let $V= \{v_1, v_2,\dotsc \}$ and $V_i = \{v_1,\dotsc, v_i\}$ for $i\in \N$.
For each $n$ we define $G_n$ to be the graph obtained by contracting each connected component $C$ of $G- V_n$ to a vertex $v_C$, referred to as \emph{auxiliary} vertices of $G_n$.
Each graph $G_n$ is finite as $e(V_n, G-V_n)$ is finite, since $V_n$ is finite and $G$ is locally finite.
Further, each $G_n$ contains a double tree $H_n$ such that $H_n[V_n] = G[V_n]$.
Indeed, let $T_1\sqcup T_2$ be a double tree decomposition for $G$.
Contracting the connected components of $G-V_n$ may create cycles. 
Since $T_1,T_2$ restricted to $V_n$ are both acyclic, each such cycle necessarily contains some $v_C$, for some connected component $C$ of $G-V_n$.
Hence, we may remove edges incident to auxiliary vertices until we obtain a double tree $H_n$. 

By Theorem~\ref{thm:2trees}, each $H_n$ has a $c$-balanced decomposition $T_1^{(n)}\sqcup T_2^{(n)}$.
By a standard compactness argument we may pass to a subsequence $(n_k)_k$ such that for every $k > \ell$ the decompositions agree on $V_{\ell}$, \textit{i.e.} 
\begin{align*}
T_1^{(n_k)}[V_{\ell}] &= T_1^{(n_\ell)}[V_{\ell}],\\
T_2^{(n_k)}[V_{\ell}] &= T_2^{(n_\ell)}[V_{\ell}].
\end{align*}
Take $S_1$ and $S_2$ to be the unions of $(T_1^{(n_k)}[V_{n_k}])_k$ and $(T_2^{(n_k)}[V_{n_k}])_k$, respectively.
Clearly, $S_1$ and $S_2$ are spanning subgraphs of $G$.
Since $G$ is locally finite, for any $v\in V(G)$ there is some $K$ such that for $k\geq K$, we have $\Gamma(v) \subset V_{n_k}$ and thus $\{ e \in E(G) \colon v\in e\} \subset E(T_1^{(n_k)})\sqcup E(T_2^{(n_k)})$.
This implies that $S_1$ and $S_2$ partition the edges of $G$ and since every decomposition $T_1^{(n_k)}\sqcup T_2^{(n_k)}$ is $c$-balanced, we conclude that $S_1 \sqcup S_2$ is $c$-balanced.

It remains to check that $S_1$ and $S_2$ intersect every finite cut of $G$.
Let $(A,B)$ be a finite cut of $G$.
Since $(A,B)$ is finite, there is some $k$ such that $E(A,B)\subset G[V_{n_k}]$.
Let $x\in A\cap V_{n_k}$ and $y\in B\cap V_{n_k}$.
Since $T_1^{n_k}$ is connected and contains $V_{n_k}$, it contains a path $P$ from $x$ to $y$. 
We claim that $P\cap E(A,B)\neq \emptyset$, finishing the proof.
Indeed, the path $P$ may be extended to a path $P'$ between $x$ and $y$ in $G$ such that $P'$ and $P$ coincide on $G[V_{n_k}]$, and whose only additional edges have endpoints outside of $V_{n_k}$.
Since $(A,B)$ is a cut of $G$, $P'\cap E(A,B)\neq \emptyset$.
But $E(A,B) \subset E(G[V_{n_k}])$ so $P\cap E(A,B)\neq \emptyset$.
Hence, $S_1$ is semiconnected.
Similarly, $S_2$ is semiconnected.
\end{proof}
This compactness argument can easily be modified to yield Theorem~\ref{thm:semiconngeneral} 
by applying Theorem~\ref{thm:2treesgeneral} instead of Theorem~\ref{thm:2trees} in the proof.

\section{Digraphs}\label{sec:digraphs}

Arborescences are the natural analogue for trees in digraphs and so H\"{o}rsch's Question~\ref{qu:digraphs} asks whether the digraph analogue of Theorem~\ref{thm:horsch} holds. A natural analogue of connectedness for digraphs is strong connectedness. The following question is then the digraph analogue of Theorem~\ref{thm:2treesgeneral}:
does any union of two strongly connected digraphs allow a balanced decomposition into two strongly connected digraphs?
We answer both this and Question~\ref{qu:digraphs} in the negative. In fact, our counterexamples have unique decompositions and these decompositions are not balanced.
\subsection{Arborescences}\label{sec:arbor}
In this subsection we answer Question \ref{qu:digraphs} in the negative.
More precisely, we show the following.
\begin{theorem}\label{conj:karb}
Let $k \geq 2$.
For every $c>0$, there is a $k$-arborescence $D$,
such that no \mbox{$k$-arborescence} decomposition $D= A_1\sqcup \dotsb \sqcup A_k$ is $c$-balanced,
meaning that for all distinct $i, j$, there is some vertex $v$ with 
\begin{equation*}
    |d^{\textup{out}}_{A_i} - d^{\textup{out}}_{A_j}| \leq c.
\end{equation*}
\end{theorem}

\begin{proof}
For $k=2$, we construct an example on vertex set $V = \{ v_1,\dotsc, v_n\}$ as in Figure~\ref{fig:arbcounter}.

\begin{figure}[ht]
	\centering
	\begin{tikzpicture}
	[scale=1.2,auto=center,circ/.style={circle,draw,outer sep=1.7, inner sep=1.7, fill=black}, >=stealth]
	\begin{scope}
	\node[circ, label={180:\small $v_1$}] (s) at (0,0) {};
	\node[circ] (v1) at (1,0) {};
	\node[circ] (v2) at (2,0) {};
	\node[circ] (v3) at (3,0) {};
	\node[circ] (v4) at (4,0) {};
	\node[circ] (v5) at (5,0) {};
	\node[circ] (v6) at (6,0) {};
	\node[circ, label={360:\small $v_n$}] (v7) at (7,0) {};
	\end{scope}
	\begin{scope}[every edge/.style={->, very thick,draw=red}]
	\draw (s) edge (v1);
	\draw (v1) edge (v2);
	\draw (v2) edge (v3);
	\draw (v3) edge (v4);
	\draw (v4) edge (v5);
	\draw (v5) edge (v6);
	\draw (v6) edge (v7);
	\end{scope}
	\begin{scope}[every edge/.style={->, very thick,draw=blue}]
	\draw (s) edge[bend left=50] (v7);
    \draw (v7) edge[bend left=50] (v6);
    \draw (v7) edge[bend left=50] (v5);
    \draw (v7) edge[bend left=50] (v4);
    \draw (v7) edge[bend left=50] (v3);
    \draw (v7) edge[bend left=50] (v2);
    \draw (v7) edge[bend left=50] (v1);
	\end{scope}
	\end{tikzpicture}
	\caption{Construction in Theorem~\ref{conj:karb} when $k=2$.}
	\label{fig:arbcounter}
\end{figure}
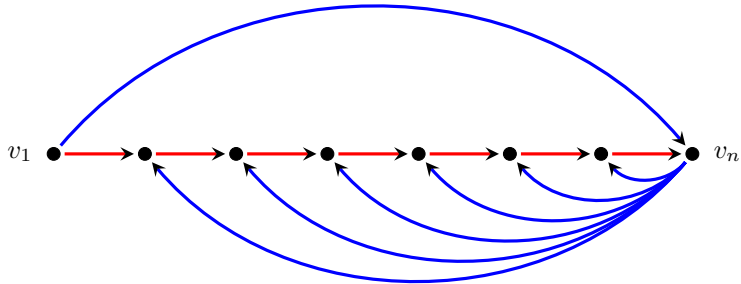

Let $B_1$ be the arborescence in blue and $B_2$ the directed path in red.
We claim that this is the unique double arborescence decomposition of the resulting digraph $D_n$, up to reordering, thus proving the result as 
\begin{equation*}
d_{B_1}^{\textup{out}}(v_n) - d_{B_2}^{\textup{out}}(v_n) = n-2.
\end{equation*}
Indeed, let $D= C_1\sqcup C_2$ be an arbitrary double arborescence decomposition of $D$.
Without loss of generality, $ \vv{v_1v_n}\in A(C_1)$.
Since $C_2$ is connected, it contains a directed path from $v_1$ to $v_n$. 
But the only such path that does not use the arc $\vv{v_1v_n}$ is the path $B_2$.
Hence, $C_2= B_2$ and $C_1 = B_1$, as claimed.

This example can easily be generalised to show Theorem~\ref{conj:karb} for general $k$, for example by adding $k-2$ copies of the directed path $B_2$.
\end{proof}

\subsection{Strongly connected digraphs}\label{sec:strongcon}
We now give the counterexample for the second question mentioned above.
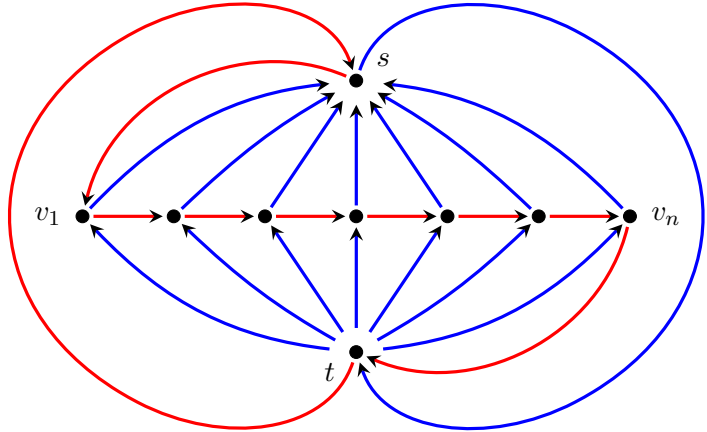
\begin{figure}[ht]
	\centering
	\begin{tikzpicture}
	[scale=1.2,auto=center,circ/.style={circle,draw,outer sep=1.7, inner sep=1.7, fill=black}, >=stealth]
	\node[circ, label={25:$s$}] (s) at (4,1.5) {};
	\node[circ, label={180:$v_1$}] (v1) at (1,0) {};
	\node[circ] (v2) at (2,0) {};
	\node[circ] (v3) at (3,0) {};
	\node[circ] (v4) at (4,0) {};
	\node[circ] (v5) at (5,0) {};
	\node[circ] (v6) at (6,0) {};
	\node[circ, label={360:$v_n$}] (v7) at (7,0) {};
	\node[circ, label={183:$t$}] (t) at (4,-1.5) {};
	\begin{scope}[circ/.style={circle,draw=white,outer sep=1.7, inner sep=1.7, fill=none}]
	\node[circ] (c1) at ( 4.3, -2.4) {};
	\node[circ] (c2) at ( 08.2, -0.95) {};
	\node[circ] (c3) at ( 5.45, 2.55) {};
	\node[circ] (d1) at ( 5.45, -2.55) {};
	\node[circ] (d2) at (08.2,  0.95) {};
	\node[circ] (d3) at ( 4.3,  2.4) {};
	\node[circ] (q1) at ( 3.7, -2.4)    {};
	\node[circ] (q2) at (-00.2, -0.95)  {};
	\node[circ] (q3) at ( 2.55, 2.55)   {};
	\node[circ] (w1) at ( 2.55, -2.55)  {};
	\node[circ] (w2) at (-00.2,  0.95)  {};
	\node[circ] (w3) at ( 3.7,  2.4)    {};
	\end{scope}
	\begin{scope}[every edge/.style={->, very thick,draw=red}]
	\draw (s) edge[bend right=50] (v1);
	\draw (v1) edge (v2);
	\draw (v2) edge (v3);
	\draw (v3) edge (v4);
	\draw (v4) edge (v5);
	\draw (v5) edge (v6);
	\draw (v6) edge (v7);
	\draw (v7) edge[bend left=50] (t);
	\draw[->, very thick,draw=red, rounded corners] 
	  (t) .. controls (q1) and (w1) .. (1.50,-2)
          .. controls (q2) and (w2) .. (1.50,2)
          .. controls (q3) and (w3) .. (s);
	\end{scope}
	\begin{scope}[every edge/.style={<-, very thick,draw=blue}]
	\draw (s) edge[shorten <=06pt,bend left=20] (v7);
    \draw (s) edge[shorten <=05pt,bend left=7] (v6);
    \draw (s) edge[shorten <=05pt] (v5);
    \draw (s) edge[shorten <=05pt] (v4);
    \draw (s) edge[shorten <=05pt] (v3);
    \draw (s) edge[shorten <=05pt, bend right=7] (v2);
    \draw (s) edge[shorten <=06pt, bend right=20] (v1);
    \draw (v1) edge[shorten >=06pt, bend right=20] (t);
    \draw (v2) edge[shorten >=05pt, bend right=7] (t);
    \draw (v3) edge[shorten >=05pt] (t);
    \draw (v4) edge[shorten >=05pt] (t);
    \draw (v5) edge[shorten >=05pt] (t);
    \draw (v6) edge[shorten >=05pt,bend left=7] (t);
    \draw (v7) edge[shorten >=06pt,bend left=20] (t);
	\draw[<-, very thick,draw=blue, rounded corners] 
	  (t) .. controls (c1) and (d1) .. (6.50,-2)
          .. controls (c2) and (d2) .. (6.50,2)
          
          .. controls (c3) and (d3) .. (s);
	\end{scope}
	\end{tikzpicture}
	\caption{Construction in Theorem~\ref{conj:kstrongdigraphs} when $k=2$.}
	\label{fig:strongconn}
\end{figure}
\begin{theorem}\label{conj:kstrongdigraphs} 
Let $k\geq 2$.
For any $c>0$, there is a digraph $D= (V,A)$ that is the union of $k$ strongly connected digraphs,
such that in any decomposition $D= S_1\sqcup \dotsb \sqcup S_k$ into strongly connected digraphs,
there is a vertex $v$ and some $i,j$ with
\begin{equation*}
|d^{\textup{out}}_{S_i}(v)- d^{\textup{out}}_{S_j}(v)| > c.
\end{equation*}
\end{theorem}
\begin{proof}
For $k=2$ we construct a family $(D_n)$ of examples.
The digraph $D_n$ has vertex set $V = \{ s,t, v_1,\dotsc, v_n\}$ as in Figure~\ref{fig:strongconn}.
Let $S_1$ and $S_2$ be the digraphs in blue and red, respectively.
It is sufficient to show that this is the unique decomposition of $D_n$ into strongly connected digraphs, as
\begin{equation*}
    |d^{\textup{out}}_{S_1}(t) - d^{\textup{out}}_{S_2}(t)| = n-1.
\end{equation*}

Let $D = R_1 \sqcup R_2$ be a decomposition of $D$ into strongly connected digraphs.
Without loss of generality, $\vv{st}\in A(R_1)$.
Since $R_2$ is strongly connected, there is a path from $s$ to $t$ in $R_2$.
The only such path that does not use $\vv{st}$ is the path $P$ with arcs $\{\vv{sv_1}, \vv{v_1v_2},\dotsc, \vv{v_{n-1}v_n}, \vv{v_nt}\}$.
Hence, all arcs of $P$ are in $R_2$.
Since $R_1$ is strongly connected, there are paths from $t$ to $v_i$ and from $v_i$ to $s$ in $R_1$, for each $i\in [n]$.
The only such arcs that are not in $P$ are the arcs $\vv{tv_i}$ and $\vv{v_is}$, respectively.
Hence, $R_1 = S_1$ and finally $R_2= S_2$, as claimed.

Similarly as for Theorem~\ref{conj:karb}, these examples can be generalised to arbitrary $k\geq 2$ by adding \(k - 2\) copies of \(S_1\).
\end{proof}

\section{Complexity}\label{sec:NP}

In this section we will show that the decision problem ``Given an Eulerian double tree, does it have a perfectly balanced double tree decomposition?'' is NP-hard, thus answering Question~\ref{conj:algo} in the negative.
We will refer to this problem as PBDT.
We need the following results.
\begin{enumerate}
\item P\'eroche \cite{peroche1984npcompleteness}: the decision problem ``Given a graph with maximum degree $4$, does it contain two edge-disjoint Hamiltonian cycles?'' is NP-complete.\footnote{This problem is referred to as `2-PAR' in the paper of P\'eroche.}
\item Roskind, Tarjan \cite{roskind1985ano}: there is an algorithm which, given a graph $G$, decides in polynomial time whether $G$ is a double tree, and outputs a double tree decomposition if it is.
\end{enumerate}
Note that if a graph contains two edge-disjoint Hamiltonian cycles, then every vertex has degree at least $4$.
So we immediately deduce from the result of P\'eroche that the decision problem ``given a $4$-regular graph, does it contain two edge-disjoint Hamiltonian cycles?'' is NP-complete.
It suffices to reduce this problem to PBDT. Let $\mathcal{A}$ be an algorithm solving PBDT.

Given a $4$-regular graph $G$, fix a vertex $v$ and let its neighbours be $v_1,v_2,v_3,v_4$. We perform the following reductions:
for $i = 1, 2, 3$, let $G_i$ be the graphs obtained by removing $v$, adding vertices $x,y$ and adding edges from $x$ to $v_1$, $v_{i + 1}$ and connecting $y$ to the other two $v_j$.
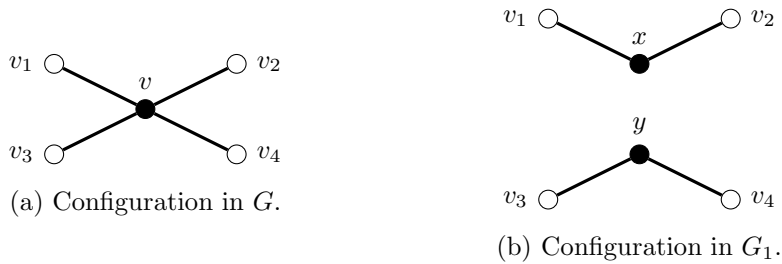
\begin{figure}[ht] 
	\centering
    \begin{subfigure}[c][70pt][c]{0.4\textwidth}
	\centering
	\begin{tikzpicture}
	[scale=.6,auto=center,circ/.style={circle,draw, minimum size=1pt, inner sep=2.5pt, fill=black}, >=stealth]
    	\node[circ, fill=black, label={\small $v$}] (v) at (0,0) {};
    	\node[circ, fill=white, label={180:\small $v_3$}] (v1) at (-2,-1) {};
    	\node[circ, fill=white, label={180:\small $v_1$}] (v2) at (-2,1) {};
    	\node[circ, fill=white, label={0:\small $v_4$}] (v3) at (2,-1) {};
    	\node[circ, fill=white, label={0:\small $v_2$}] (x) at (2,1) {};
    	\begin{scope}[every edge/.style={very thick,draw=black}]
    	\draw (v1) edge (v);
    	\draw (v2) edge (v);
    	\draw (v) edge (v3);
    	\draw (v) edge (x);
    	\end{scope}
	\end{tikzpicture}
	\caption{Configuration in $G$.}
    \end{subfigure} 
    \begin{subfigure}[c][90pt][c]{0.4\textwidth}
	\centering
	\begin{tikzpicture}
	[scale=.6,auto=center,circ/.style={circle,draw, minimum size=1pt, inner sep=2.5pt, fill=black}, >=stealth]
    	\node[circ, fill=black, label={\small $x$}] (x) at (0,1) {};
    	\node[circ, fill=black, label={\small $y$}] (y) at (0,-1) {};
    	\node[circ, fill=white, label={180:\small $v_1$}] (v1) at (-2,2) {};
    	\node[circ, fill=white, label={0:\small $v_2$}] (v2) at (2,2) {};
    	\node[circ, fill=white, label={180:\small $v_3$}] (v3) at (-2,-2) {};
    	\node[circ, fill=white, label={0:\small $v_4$}] (v4) at (2,-2) {};
    	\begin{scope}[every edge/.style={very thick,draw=black}]
    	\draw (v1) edge (x);
    	\draw (v2) edge (x);
    	\draw (y) edge (v3);
    	\draw (y) edge (v4);
    	\end{scope}
	\end{tikzpicture}
	\caption{Configuration in $G_1$.}
    \end{subfigure} 
\end{figure}

For $i=1,2,3$, run the algorithm of Roskind and Tarjan. 
If it outputs a double tree decomposition for $G_i$, 
run $\mathcal{A}$ on it.
\begin{claim}
The graph $G$ contains two disjoint Hamiltonian cycles if and only if one of $G_1,G_2,G_3$ has a perfectly balanced double tree decomposition.
\end{claim}
\begin{proof}
Note that for $i=1,2,3$, if $G_i$ has a perfectly balanced decomposition $T_1\sqcup T_2$ then $T_1$, $T_2$ are two edge-disjoint Hamiltonian paths with endpoints $x$ and $y$.
Indeed, every $v\in V(G_i)\setminus \{x,y\}$ must have degree $2$ in each tree and $x,y$ must be leaves. 

Therefore, a perfectly balanced double tree decomposition in $G_i$ corresponds to two edge-disjoint Hamiltonian cycles in $G$ by merging vertices $x$ and $y$. 
Conversely, two edge-disjoint Hamiltonian cycles in $G$ yield a perfectly balanced decomposition in one of the three splittings of $v$ into $x$ and $y$ described above.
\end{proof}

Hence, the above algorithm is a valid polynomial time reduction of finding two edge-disjoint Hamiltonian cycles in a $4$-regular graph to PBDT.

\section{Conclusion}\label{sec:conc}
We have shown that every double tree has a partition into two trees such that the degrees at each vertex differ by at most four (improving on H\"{o}rsch's~\cite{horsch2021globally} bound of five). Can this be further improved? There are examples of double trees that admit a $2$-balanced double tree decomposition, but no $1$-balanced double tree decomposition.
The only such examples known to the authors involve taking an odd cycle, 
whose edges cannot be colored blue/red without creating a vertex with degree difference $2$, 
and making it into a double tree while preserving degree differences.
See below for example.
In any double tree decomposition, one of the vertices of the triangle has degree difference $2$.
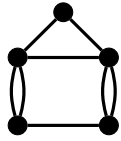
\begin{figure}[ht] 
    \centering
	\begin{tikzpicture}
	[scale=.6,auto=center,circ/.style={circle,draw, minimum size=1pt, inner sep=2.5pt, fill=black}, >=stealth]
    	\node[circ] (v) at (0,0) {};
    	\node[circ] (s) at (1,-1) {};
    	\node[circ] (u) at (-1,-1) {};
    	\node[circ] (u1) at (-1,-2.5) {};
    	\node[circ] (s1) at (1,-2.5) {};
    	\begin{scope}[every edge/.style={very thick,draw=black}]
    	\draw (v) edge (s);
    	\draw (v) edge (u);
    	\draw (s1) edge[bend left=15] (s);
    	\draw (s1) edge[bend right=15] (s);
    	\draw (u) edge (s);
    	\draw (u) edge[bend right=15] (u1);
    	\draw (u) edge[bend left=15] (u1);
    	\draw (u1) edge (s1);
    	\end{scope}
	\end{tikzpicture}
\end{figure}

It seems natural to conjecture that this lower bound is tight.
\begin{conjecture}
Any double tree has a $2$-balanced double tree decomposition.
\end{conjecture}

The question of balancing double trees can naturally be generalised to balancing $k$-trees, as well as graphs containing $k$ edge-disjoint trees.
\begin{question}\label{qu:general}
Let $k\geq 2$.
What are the smallest constants $c_k, d_k>0$ such that the following hold?
\begin{itemize}
\item Any finite graph which is the union of $k$ edge-disjoint spanning trees has a \mbox{$c_k$-balanced} $k$-tree decomposition.
\item Any finite graph  
containing $k$ edge-disjoint spanning trees has a $d_k$-balanced decomposition into connected spanning subgraphs. 
\end{itemize}
\end{question}
By repeatedly applying Theorem~\ref{thm:horsch}, H\"orsch~\cite{horsch2021globally} obtained the bound \(c_k \leq 16\log k\). 
We could similarly derive improved bounds on \(c_k\) and \(d_k\) by repeatedly applying Theorem~\ref{thm:2treesgeneral}. 
When the requirement that each graph in the decomposition is a tree is dropped (so any \(k\)-edge colouring of the original graph \(G\) is allowed), a uniform bound on the colour-degree differences is attainable. 
Indeed, let \(\mathcal{H}\) be the hypergraph whose vertices are the edges of \(G\) and whose hyperedges are the stars centred at each vertex of \(G\). 
Then \(\mathcal{H}\) has maximum degree \(2\) and so bounded discrepancy -- see, for example, the paper of Doerr and Srivastav~\cite[Theorem~3.7]{ds2003}. In particular, upper bounds on $c_k$ are bounds on $d_k$ with a constant error term.
It would be particularly interesting to resolve H\"{o}rsch's conjecture~\cite{horsch2021globally} of whether there is a uniform upper bound on the \(c_k\).

The digraphs used for the proofs of Theorems~\ref{conj:karb} and \ref{conj:kstrongdigraphs} in Section~\ref{sec:digraphs} rely on the uniqueness of the decompositions into arborescences/strongly connected digraphs.
It is natural to ask what happens if our starting digraph is less restricted.
\begin{question}
Are there constants $c,t$ such that if $D$ is a disjoint union of $t$ spanning arborescences sharing a root, then the edges of $D$ can be coloured blue/red such that the out-degrees are $c$-balanced and both graphs contain arborescences?
\end{question}
The same question is also interesting for strongly connected digraphs. The hypothesis that $D$ is a disjoint union of $t$ strongly connected spanning digraphs is slightly cumbersome and it would seem natural to replace it with some high connectivity condition. As far as we are aware, the following question is open and would be interesting to resolve.

\begin{question}
For each positive integer \(t\) is there a constant \(k\) such that the edges of any \(k\)-strongly connected digraph can be partitioned into \(t\) parts each of which is spanning and strongly connected?
\end{question}

For undirected graphs the corresponding statement follows from the Tutte-Nash-Williams characterisation with \(k = 4t\).

\section{Acknowledgements}

The authors would like to thank Lex Schrijver for helpful comments.

\bibliographystyle{scott}
\bibliography{treepartitions}

\begin{thebibliography}{10}

\bibitem{bang2016complexity}
J.~Bang-Jensen, F.~Havet and A.~Yeo.
\newblock The complexity of finding arc-disjoint branching flows.
\newblock {\em Discrete Applied Mathematics}, {\bfseries 209}:16--26, 2016.

\bibitem{bessy2021fpt}
S.~Bessy, F.~H{\"o}rsch, A.~K. Maia, D.~Rautenbach and I.~Sau.
\newblock {FPT} algorithms for packing $ k $-safe spanning rooted sub (di)
  graphs.
\newblock {\em arXiv:2105.01582}, 2021.

\bibitem{chuzhoy2020packing}
J.~Chuzhoy, M.~Parter and Z.~Tan.
\newblock On packing low-diameter spanning trees.
\newblock {\em arXiv:2006.07486}, 2020.

\bibitem{ds2003}
B.~Doerr and A.~Srivastav.
\newblock Multicolour discrepancies.
\newblock {\em Combinatorics, Probability and Computing}, {\bfseries
  12}:365--399, 2003.

\bibitem{horsch2021globally}
F.~H\"orsch.
\newblock Globally balancing spanning trees.
\newblock {\em arXiv:2110.13726}, 2021.

\bibitem{kriesell2011balancing}
M.~Kriesell.
\newblock Balancing two spanning trees.
\newblock {\em Networks}, {\bfseries 57}(4):351--353, 2011.

\bibitem{nash1961edge}
C.~Nash-Williams.
\newblock Edge-disjoint spanning trees of finite graphs.
\newblock {\em Journal of the London Mathematical Society}, {\bfseries
  1}(1):445--450, 1961.

\bibitem{oxley1979packing}
J.~G. Oxley.
\newblock On a packing problem for infinite graphs and independence spaces.
\newblock {\em Journal of Combinatorial Theory, Series B}, {\bfseries
  26}(2):123--130, 1979.

\bibitem{peroche1984npcompleteness}
B.~P{\'e}roche.
\newblock {NP}-completeness of some problems of partitioning and covering in
  graphs.
\newblock {\em Discrete Applied Mathematics}, {\bfseries 8}:195--208, 1984.

\bibitem{roskind1985ano}
J.~Roskind and R.~E. Tarjan.
\newblock A note on finding minimum-cost edge-disjoint spanning trees.
\newblock {\em Mathematics of Operations Research}, {\bfseries 10}:701--708,
  1985.

\bibitem{tutte1961problem}
W.~T. Tutte.
\newblock On the problem of decomposing a graph into $n$ connected factors.
\newblock {\em Journal of the London Mathematical Society}, {\bfseries
  1}(1):221--230, 1961.

\end{thebibliography}

\end{document}